\documentclass[review]{elsarticle}
\usepackage{latexsym, graphicx, epsfig, amsmath, amsfonts,amssymb,subfigure,algorithm,algorithmicx,booktabs,bm,color,bigints,}
\usepackage{epstopdf}
\usepackage{amsthm}

\def\e{{\epsilon}}

\def\i{\mathbf{i}}

\def\E{\mathbb{E}}
\def\w{{\omega}}

\def\mb{\mathbf}

\newtheorem{remark}{Remark}
\newtheorem{theorem}{Theorem}
\newtheorem{lemma}{Lemma}

\begin{document}

\begin{frontmatter}

\title{Gaussian Process Regression and Conditional Polynomial Chaos for Parameter Estimation}

\author[mymainaddress]{Jing Li\corref{mycorrespondingauthor}}
\cortext[mycorrespondingauthor]{Corresponding author}
\ead{jing.li@pnnl.gov}

\author[mymainaddress]{Alexandre Tartakovsky}
\ead{alexandre.tartakovsky@pnnl.gov}
\address[mymainaddress]{Pacific Northwest National Laboratory, Richland, WA 99352.}%

\begin{abstract}
We present a new approach for constructing a data-driven surrogate model and using it for Bayesian parameter estimation in partial differential equation (PDE) models. We first use parameter observations and Gaussian Process regression to condition the Karhunen-Lo\'{e}ve (KL) expansion of the unknown space-dependent parameters and then build the conditional generalized Polynomial Chaos (gPC) surrogate model of the PDE states. Next, we estimate the unknown parameters by computing coefficients in the KL expansion minimizing  the square difference between the gPC predictions and measurements of the states using the Markov Chain Monte Carlo method. 
Our approach addresses two major challenges in the Bayesian parameter estimation. First, it reduces dimensionality of the parameter space and replaces expensive direct solutions of PDEs with the conditional gPC surrogates. Second, the estimated parameter field exactly matches the parameter measurements. In addition, we show that the conditional gPC surrogate can be used to estimate the states variance, which, in turn, can be used to guide data acquisition. We demonstrate that our approach improves its accuracy with application to one- and two-dimensional Darcy equation with (unknown) space-dependent hydraulic conductivity. We also discuss the effect of hydraulic conductivity and head locations on the accuracy of the hydraulic conductivity estimations.

\end{abstract}


\end{frontmatter}

\section{Introduction}
Here, we focus on inverse problems, namely, parameter estimation in partial differential equations (PDEs) models. We are motivated by applications with high-dimensional parameter space.  These problems are usually ill posed and deterministic parameter estimation methods are computationaly expensive and require regularization to obtain unique solutions \cite{beck1985inverse, Cheney1990inverse, mclaughlin1996reassessment, Hanke_1997}. As an alternative, Bayesian methods have been proposed for obtaining the most probable combination of the parameters \cite{mclaughlin1996reassessment,stuart15c}. Among Bayesian methods, the Markov chain Monte Carlo method provides a powerful tool for generating the posterior distributions of the parameters \cite{gamerman2006MCMC} and finding the most probable point in the parameter space for the ``optimal'' choice of parameters. However, MCMC requires a large number of forward solutions of the PDE model to generate posterior distributions that could  be computationally expensive. Therefore, accurate and computational efficient surrogates for the PDE model are often sought \cite{Christen2005}. 

Various surrogates have been adopted to accelerate the Bayesian inference for inverse problems \cite{Frangos2010,Nguyen2014,Cui2015, Chen2016}.  For example, the Gaussian process regression (GPR) or Kriging model was  used in \cite{JJZhang2016} to approximate the relation between parameters and the forward solutions. In \cite{MARZOUK2007, MARZOUK2009, MarzoukXiu2009}, the generalized polynomial chaos (gPC) expansion was employed to approximate the forward solution of the PDE models as a function of the random parameters in the  Karhunen-Lo\'{e}ve (KL) expansion of the unknown coefficient. 

In this work, we propose a novel gPC-based surrogate model, which we call the conditional gPC surrogate model, and use it to approximate the solution of a PDE model in the MCMC approach for parameter estimation. The easy-to-evaluate conditional gPC surrogate significantly accelerates the MCMC sampling and the parameter estimation and quarantees that the estimated parameters exactly match the parameter measurements. The conditional gPC is based on the data-driven conditional KL representation of unknown space-dependent parameters \cite{ossiander2014,lih_2014,ZhangDX_KL} that reduces the dimensionality of the parameter space and allows parameter estimation with relatively few measurements. There are two levels of the dimension reduction in the conditional KL and gPC methods. First, the exact infinite KL representation of a random parameter is approximated with a finite KL  expansion \cite{Loeve1978,GhanemS91,Schwab2006,Xiu10}. At the second level,  available observations of the parameter are used to condition the truncated KL expansion and further reduce the dimensionality. The reduction of dimensionality by conditioning on the parameter measurements partially alleviates the curse of dimensionality and allows  a high-order gPC construction with relatively few samples (collocation points).

We use the conditional gPC surrogate model in combination with MCMC to estimate a partially observed space-dependent diffusion coefficient in the steady-state diffusion equation given partial observations of the state variables. 
We model the unknown parameters and states as random fields and represent the parameters with the conditional KL expansion. The partial measurements of the (random) state variable are treated as measurements of one realization of this variable and a minimization problem is formed to estimate the parameter by computing the optimal coefficients (values of the random variables) in the conditional KL expansion. Furthermore, we evaluate several sampling strategies of the states. We find that the gPC surrogate is more sensitive with respect to the parameters at the locations where it has higher variance. Therefore, measuring the state variable where its variance is largest allows more accurate parameter estimation as compared to uniformly and randomly distributed state measurements. Since in practice only few measurements of state are available, the ``optimal'' choice of measurement locations is important. Moreover, this strategy can be extended to the other regression approaches under the probabilistic framework. 

The remainder of this paper is organized as follows. In Section ~\ref{sec:opt_form}, we formulate the spatial dependent coefficient estimation problem via the parameter minimization problem. The conditional KL representation with reduced dimensionality is stated in Section ~\ref{sec:kappa_representation}. In Section ~\ref{sec:cond_gPC}, we construct the conditional gPC surrogate and analyze its properties. In Section ~\ref{sec:inverse_problem}, we reformulate our minimization problem under the Bayesian framework and briefly introduce the MCMC method for parameter estimation. We study the accuracy of our approach through numerical examples, including parameter estimation in one- and two-dimensional elliptic PDEs in Section ~\ref{sec:num}. Conclusions are given in Section ~\ref{sec:conclusion}.

\section{Problem Formulation}\label{sec:opt_form}

Consider the (deterministic) steady-state diffusion equation with 
space-dependent partially known coefficient $\hat{\kappa}(\mathbf{x})$ 
and appropriate boundary conditions:

\begin{align}\label{eq:diffusion-deterministic}
\begin{split}
-\nabla \cdot (\hat{\kappa}(\mb{x})\nabla \hat{u}(\mb{x})) &= 0, \qquad \mb{x}\in D;\\
 \hat{u}(\mb{x}) &= f(\mb{x}),\qquad \mb{x} \in \partial D_D;\\
 \vec{n}\cdot \hat{\kappa}(\mb{x})\nabla  \hat{u}(\mb{x})&= g(\mb{x}), \qquad \mb{x} \in \partial D_L.
\end{split}
 \end{align}
Among other applications, Eq (\ref{eq:diffusion-deterministic}) describes flow in geological porous media, where $\hat{\kappa}(\mb{x})$ is the hydraulic conductivity and $\hat{u}(\mb{x})$ is the hydraulic head. In this application, for financial and technical reasons, $\hat{\kappa}(\mb{x})$ only can be measured in a few locations. 
Numerical treatments of Eq (\ref{eq:diffusion-deterministic}) are based on the idea that $\hat{\kappa}(\mb{x})$ can be resolved by a finite vector 
$\hat{\bm{\kappa}}:=\left(\hat{\kappa}(\mb{x}_1),\dots,\hat{\kappa}(\mb{x}_n)\right)$ consisting of $\hat{\kappa}$ values 
 at the collection of points $\{\mb{x}_i\}_{i=1}^n \subset D$. 

We employ the probabilistic approach where we treat (the partially known) conductivity as a random field with prior distribution learned from the $\hat{k}$ measurements. Specifically, we employ the probability space $(\Omega,\mathcal{F},P)$ and assume that $\hat{\kappa}(\mb{x})$ is a realization of a spatially heterogenous random field, i.e., there exists a $P$-measurable map $\kappa(\cdot,\omega):\Omega \rightarrow L^{\infty}(D)$ and $\hat{\kappa}(\mb{x}) = \kappa(\mb{x},\omega^*)$. 

This probabilistic approach renders Eq (\ref{eq:diffusion-deterministic}) stochastic:  

\begin{align}\label{eq:diffusion}
\begin{split}
\nabla \cdot (\kappa(\mb{x},\omega)\nabla u(\mb{x},\omega)) &= 0, \qquad \mb{x}\in D;\\
 u(\mb{x},\omega) &= f(\mb{x}),\qquad \mb{x} \in \partial D_D;\\
 \vec{n}\cdot \kappa(\mb{x},\omega)\nabla  u(\mb{x},\omega)&= g(\mb{x}), \qquad \mb{x} \in \partial D_L.
\end{split}
 \end{align}

We propose to estimate $\hat{\kappa}(\mb{x}) = \kappa(\mb{x},\omega^*)$ given partial measurements $\{\hat{\kappa}(\mb{x}^{(i)})\}_{i=1}^{N_m}$ and $\{\hat{u}(\mb{x}^{(j)})\}_{j=1}^{N_k}$ by solving the following minimization problem, 
\begin{equation}\label{eq:opt_o_full}
{\omega}^* = \textrm{argmin}_{\omega\in\Omega} \left[ \sum_{j=1}^{N_k} |\hat{u}(\mb{x}^j)-u(\mb{x}^{j},\omega)|^2
+ \sum_{j=1}^{N_m} |\hat{k}(\mb{x}^j)-k(\mb{x}^{j},\omega)|^2 \right].
\end{equation}

Solving this ``full'' parameter estimation problem with Bayesian inference is challenging specifically, it is challenging  to minimize both terms in Eq (\ref{eq:opt_o_full}). Instead, it is common to solve a simpler minimization problem: 
\begin{equation}\label{eq:opt_o}
{\omega}^* = \textrm{argmin}_{\omega\in\Omega}\sum_{j=1}^{N_k} |\hat{u}(\mb{x}^j)-u(\mb{x}^{j},\omega)|^2,
\end{equation}
where $\hat{\kappa}$ measurements are only used to obtain a prior distribution of $k(\mb{x},\omega)$ \cite{MARZOUK2007, MARZOUK2009, MarzoukXiu2009}. Therefore, $\kappa(\mb{x},\omega^*)$ found from the optimization problem (\ref{eq:opt_o}) does not guarantee to much $\hat{\kappa}$ observations, while one found from (\ref{eq:opt_o_full}) does. 
Following common practice, we assume that the conductivity $\kappa(\mb{x}, \omega)$ in Eq (\ref{eq:diffusion}) has a lognormal distribution, i.e., $Y(\mb{x},\omega) = \ln \kappa(\mb{x}, \omega)$ has normal distribution \cite{de1986quantitative}. 
In the work of \cite{MARZOUK2007, MARZOUK2009, MarzoukXiu2009}, the ``unconditional'' KL expansion of  $Y(\mb{x},\omega)$ and gPC representation of $u(\mb{x},\omega)$ in Eq (\ref{eq:opt_o}) were used. 
Here, we use a ``conditional'' KL representation of $Y(\mb{x},\omega)$ and conditional gPC stochastic collocation method to construct a surrogate for $u(\mb{x},\omega)$ and solve the optimization problem (\ref{eq:opt_o_full}). The conditional surrogate has two main advantages: it reduces the dimensionality of $\Omega$; and it exactly satisfies the second term in (\ref{eq:opt_o_full}), i.e., $\hat{k}(\mb{x}^j) \equiv \tilde{k}(\mb{x}^{j},\omega)$, where $\tilde{k}(\mb{x}^{j},\omega)$ is the KL expansion of $\kappa$ conditioned on $\{\hat{\kappa}(\mb{x}^{(i)})\}_{i=1}^{N_m}$. The latter reduces the optimization problem  (\ref{eq:opt_o_full}) to the simpler optimization problem  (\ref{eq:opt_o}). Both these advantages of the conditional gPC surrogate reduce the computational cost and improve accuracy of parameter estimation. The conditional KL and gPC methods are described in the following two sections. In Section \ref{sec:inverse_problem}, we present a method for solving the minimization problem (\ref{eq:opt_o}) and demonstrate that its dimensionality is significantly smaller than those based on the unconditional KL representation of $Y(\mb{x},\omega)$.   
 
\section{Karhunen-Lo\`{e}ve representation of $Y(\mb{x},\omega)$}\label{sec:kappa_representation}
KL representation of random coefficients combined with the gPC collocation method is the most common way to solve the stochastic PDE \eqref{eq:diffusion}. 
The main challenge in this approach is that its computational cost exponentially increases with the number of terms in the truncated KL expansion.    
Until recently, gPC methods had been only applied to stochastic PDEs with random coefficients modeled as second-order stationary random fields, i.e., fields with constant variances and covariance functions depending only on the distance between two points. The number of terms in the truncated KL expansion of the stationary covariance function  depends on the correlation length (it increases with the decreasing correlation length). Here, we condition the KL expansion of $Y(\mb{x},\omega)$ on the measurement $\{\ln\hat{\kappa}(\mb{x}^{(i)})\}_{i=1}^{N_m}$. The ``conditional'' KL representation of $Y(\mb{x},\omega)$ has zero variance at 
$\{\mb{x}^{(i)}\}_{i=1}^{N_m}$
 (for simplicity, we assume that the measurements are exact) and increases away from the measurement locations. In this section, we demonstrate that the dimensionality of conditional KL expansion is smaller than that of the corresponding unconditional KL expansion. In the following two sections, we  introduce the unconditional and conditional KL expansions of   $Y(\mb{x},\omega)$. 


\subsection{``Unconditional'' Karhunen-Lo\`{e}ve representation of $Y(\mb{x},\omega)$}\label{sec:random_field}
For a mean-square continuous stochastic process $Y(\mb{x},\omega) = \E[Y(\mb{x},\omega)]+Y'(\mb{x},\omega)$ with the mean (expectation with respect to $\omega$) $\overline{Y}(\mb{x})=\E[Y(\mb{x},\omega)]$, zero-mean fluctuations $Y'(\mb{x},\omega)$, and ``unconditional'' covariance function $C_Y(\mb{x},\mb{y}):= \E[Y'(\mb{x},\omega)Y'(\mb{y},\omega)]$, $\mb{x},\mb{y}\in D$, the KL expansion is given by Mercer's theorem \cite{mercer1909xvi} as
\begin{equation}\label{def:KL}
Y(\mb{x},\omega) =  \overline{Y}(\mb{x}) + \sum_{n=1}^{\infty}\sqrt{\lambda_n}\xi_n(\omega)\e_{n}(\mb{x}), \quad \textrm{in } L^2(\Omega), 
\end{equation}
where $\lambda_n$ are positive eigenvalues and $\e_n(\mb{x})$ are mutually orthogonal eigenfunctions, i.e., $\int_D \e_n(\mb{x})\e_l(\mb{x})d\mb{x} = \delta_{n,l}$ and $\delta_{n,l}$ is the Kronecker delta function. The eigenfunctions are found as the solution of the Fredholm integral equation of the second kind:
\begin{equation}\label{def:e_i}
\lambda \e(\mb{x}) = \int_D C_Y(\mb{x},\mb{y})\e(\mb{y}) d\mb{y}.
\end{equation}
The random variables $\xi_n$ are mutually uncorrelated and have zero mean and unit variance.
Moreover, if $Y(\mb{x},\omega)$ is a Gaussian process, then $\xi_n$ are independent standard normal random variables. 
By convention, the eigenvalues $\lambda_n$ in the KL expansion are arranged in the decreasing order and the truncated $N$-term KL expansion of $Y(\mb{x},\omega)$ is defined as:
\begin{equation}\label{def:truncated_KL}
Y_N(\mb{x},\w) := \overline{Y}(\mb{x})+\sum_{n=1}^N\sqrt{\lambda_n}\e_n(\mb{x})\xi_n(\omega).
\end{equation}

\subsection{``Conditional'' Karhunen-Lo\`{e}ve representation of $Y(\mb{x},\omega)$}
 The observations $\{Y(\mb{x}^{(i)}) = \ln \kappa(\mb{x}^{(i)}) \}_{i=1}^{N_m}$ are used to compute the (unconditional) stationary covariance describing an unconditional random field $Y$ through the variagram analysis or by maximizing the likelihood function\cite{Cressie1993}. In addition, the observations can be used to model a conditional random field using GPR as presented in \cite{ossiander2014}. In the following, we will use $\mb{x}^*$ to denote the set of $N_m$ observation locations $\{\mb{x}^{(i)}\}_{i=1}^{N_m}$ of the random field and $Y(\mb{x}^*)$ to denote the column vector of $N_m$ observations $\{Y(\mb{x}^{(i)})\}_{i=1}^{N_m}$.
Next, we define the covariance matrix of the $Y$ observations:
\begin{equation}
\Sigma_{i,j} = C_Y(\mb{x}^{(i)},\mb{x}^{(j)}) = \sum_{n=1}^{\infty} \lambda_n \e_n(\mb{x}^{(i)})\e_n(\mb{x}^{(j)}), \quad i,j = 1,\dots, N_m,
\end{equation}
and $R$ as the matrix with $n$th row given by the values of the $n$th eigenfunction at the observed locations $\mb{x}^*$,
\[
\e_n(\mathbf{x^*}) = (\e_n(\mb{x}^{(1)}), \dots, \e_n(\mb{x}^{(N_m)})).
\]
Next, we approximate  $Y(\mb{x},\w)$ with the truncated KL expansion with $N_G$ terms $Y(\mb{x},\w) \approx Y_{N_G}(\mb{x},\w) =\overline{Y}(\mb{x})+\sum_{n=1}^{N_G}\sqrt{\lambda_n}\e_n(\mb{x})\xi_n$.  Then, $R$ is a $N_G\times N_m$ matrix and $\Sigma = R^T\Lambda R$, where $\Lambda$ is the diagonal matrix with $(\lambda_1,\dots,\lambda_{N_G})^T$ being the positive diagonal vector. 

In \cite{ossiander2014}, it was demonstrated that $Y_{N_G}(\mb{x},\w)$ conditioned on  
$Y(\mb{x}^*) $, $\widetilde{Y}_{N_G}(\mb{x},\w) := [Y_{N_G}(\mb{x},\w)|Y(\mb{x}^*)]$, can be written in term of 
$\xi_n$ conditioned on 
$Y(\mb{x}^*) $,
$\widetilde{\xi_n}:=[\xi_n|Y(\mathbf{x}^*)]$, $n=1,...,N_G$.
Noting that the covariance between $\xi_n$ and observation $Y(\mb{x}^{(i)})$ is:
\[
C_{\xi_n,Y}(\mb{x}^{(i)}) = \sqrt{\lambda_n}\e_n(\mb{x}^{(i)}),
\]
 the conditional  mean $\tilde{\mu}$ and covariance $\widetilde{M} = (\widetilde{m}_{n,k})\in \mathbb{R}^{N_G\times N_G}$ of $\{ \widetilde{\xi}_n\}_{n=1}^{N_G}$ is obtained in the form  \cite{Tong1990}:
\begin{eqnarray}
\tilde{\mu}_n &= &\E[\widetilde{\xi}_n] = \sqrt{\lambda_n}\e_n(\mathbf{x}^*) \Sigma^{-1}({Y}(\mathbf{x}^*)-\overline{Y}(\mb{x}^*))\\
\widetilde{m}_{n,k} &=& C_{\widetilde{\xi}_n, \widetilde{\xi}_k} = \delta_{n,k}-\sqrt{\lambda_k}\e_k(\mb{x}^*)\Sigma^{-1}\e_n(\mb{x}^*)\sqrt{\lambda_n}.
\end{eqnarray}
Then,  
 $\{\widetilde{\xi}_n\}_{n=1}^{N_G}$ has the same distribution as 
\[ 
\tilde{\bm{\eta}} = \tilde{\bm{\mu}}+ \widetilde{M}\bm{\eta},
\]
where $\tilde{\bm{\mu}}=\Lambda^{1/2}R\Sigma^{-1}({Y}(\mathbf{x}^*)-\overline{Y(\mb{x}^*)})$,  $\widetilde{M} = I-\Lambda^{1/2}R\Sigma^{-1}R^T\Lambda^{1/2}$, and $\{\eta_k\}_{k=1}^{N_G}$ is a sequence of $i.i.d.$ $\mathcal{N}(0,1)$ random variables. As a result, $\tilde{Y}_{N_G}(\mb{x},\omega)$ can be written in terms of  $\{\eta_k\}_{k=1}^{N_G}$ as
\begin{equation}\label{def:cond_rf}
\begin{split}
\widetilde{Y}_{N_G}(\mb{x},\w) &= \overline{Y}(\mb{x}) +\sum_{n=1}^{N_G}\sqrt{\lambda_n}\e_n(\mb{x})\tilde{\eta}_n\\
&=\overline{Y}(\mb{x})+\sum_{n=1}^{N_G}\sqrt{\lambda_n}\e_n(\mb{x})\tilde{\mu}_n+\sum_{n=1}^{N_G}\sqrt{\lambda_n}\e_n(\mb{x})\sum_{l=1}^{N_G} m_{n,l}\eta_l.
\end{split}
\end{equation}
Note that $\Sigma$ should be a full rank matrix and $rank(\Sigma) = N_m$. If $\Sigma$ is not full rank, then it is possible to select a subset of the $N_m$ measurements for which  $\Sigma$ would be full rank.

Below,  we prove a lemma and a theorem that the rank of $\widetilde{M}$ and the dimension of $\widetilde{Y}_{N_G}(\mb{x},\w)$ could be reduced to $N_G - N_M$. 
\begin{lemma}\label{le:rankM}
If $rank(\Sigma)=N_m$, then $rank(\widetilde{M}) = N_G-N_m$.
\end{lemma}
\begin{proof} 
By definition, it is easy to verify that $M^2 = M$ and $M^T = M$. Since $M-M^2 = M(I-M) = 0$, the dimension of the null space of $M$ equals to the dimension of the range space of $I-M$ which is the same as the rank of $I-M$, i.e., $dim(ker(M)) = rank(I-M)$.
Since $I-M = \Lambda^{1/2}R\Sigma^{-1}R^T\Lambda^{1/2}$, then $rank(I-M)\leq rank(\Sigma) = N_m$. 

Also, 
$$M\Lambda^{1/2}R = \Lambda^{1/2}R - \Lambda^{1/2}R\Sigma^{-1}R^T\Lambda^{1/2}\Lambda^{1/2}R = 0.$$
Then, $dim(ker(M))\geq rank(\Lambda^{1/2}R)$. Since $\Sigma = R^T\Lambda R = (\Lambda^{1/2}R)^T(\Lambda^{1/2}R)$ and $rank(\Sigma) = N_m$, then $rank(\Lambda^{1/2}R) \geq N_m$. Thus, we have $dim(ker(M))\geq N_m$. In combination with the previous result, we get $dim(ker(M)) = N_m$ and $rank(M) = N_G-dim(ker(M)) = N_G-N_m$.
\end{proof}

\begin{theorem}\label{the:cond_rf}
Assume $\e_n(\mb{x})$ are orthonormal functions, $\lambda_n$ are positive values, $\xi_n$ are i.i.d standard normal random variables for $n=1,\dots,N_G$ and $rank(\Sigma) = N_m$ with $0<N_m<N_G$, then the conditional random field $\tilde{Y}_{N_G}(\mb{x},\w)$ can be represented by an expansion of $N_G-N_m$ i.i.d. standard normal random variables.
\end{theorem}
\begin{proof}
From Eq \eqref{def:cond_rf}, the covariance matrix of $\widetilde{Y}_{N_G}(\mb{x},\w)$, $C_{\tilde{Y}}$, can be written as
\begin{equation}\label{def:cond_cov}
\begin{split}
C_{\tilde{Y}}(\mb{x},\mb{y}) &= \E[(\tilde{Y}_{N_G}(\mb{x},\w)-\E(\tilde{Y}_{N_G}(\mb{x},\w)))(\tilde{Y}_{N_G}(\mb{y},\w)-\E(\tilde{Y}_{N_G}(\mb{y},\w)))]\\
&=\E[\sum_{n=1}^{N_G}\sqrt{\lambda_n}\e_n(\mb{x})\sum_{l=1}^{N_G} \tilde{m}_{n,l}\eta_l \sum_{s=1}^{N_G}\sqrt{\lambda_s}\e_s(\mb{y})\sum_{t=1}^{N_G} \tilde{m}_{s,t}\eta_t]\\
&=\sum_{n,s,l=1}^{N_G}\sqrt{\lambda_n}\e_n(\mb{x})\tilde{m}_{n,l}\sqrt{\lambda_s}\e_s(\mb{y}) \tilde{m}_{s,l}\\
&=\bm{\e}\Lambda^{1/2}\tilde{M} \tilde{M}^T\Lambda^{1/2}\bm{\e}^T\\
&=\bm{\e}\Lambda^{1/2}\tilde{M}\Lambda^{1/2}\bm{\e}^T
\end{split}
\end{equation}
Using Lemma \ref{le:rankM}, we write $rank(Cov_{\tilde{Y}}) = rank(\widetilde{M}) = N_G-N_m$. Consequently, there are only $N_G-N_m$ nonzero eigenvalues in the KL expansion of $\tilde{Y}_{N_G}(\mb{x},\w)$ that can be expressed as
\begin{equation}\label{def:cond_rf_KL}
\tilde{Y}_{N_G}(\mb{x},\w) \sim \overline{Y}(\mb{x})+\sum_{n=1}^{N_G}\sqrt{\lambda_n}\e_n(\mb{x})\tilde{\mu}_n+\sum_{i=1}^{N_G-N_m}\sqrt{\tilde{\lambda}_i}\tilde{\e}_i(\mb{x})\xi_i,
\end{equation}
where $\tilde{\lambda}_i$ and $\tilde{\e}_i(\mb{x})$, $i=1,\dots,N_G-N_m$ are eigenvalues and eigenfunctions of $Cov_{\tilde{Y}}$ respectively, and $\xi_i$, $i=1,\dots,N_G-N_m$ are i.i.d. standard Gaussian random variables.
\end{proof}
\begin{remark}
The dimensionality of the random field $\tilde{Y}_{N_G}(\mb{x},\w)$ is $N_G-N_m$ and is smaller than that of the original unconditioned random field, i.e., $N_G$.
Conditioning on observation reduces the complexity of the forward problem \eqref{eq:diffusion-deterministic}.
In addition to the dimensionality reduction, conditioning also reduces the variance of $\tilde{Y}_{N_G}(\mb{x},\w)$ \cite{ossiander2014}.
\end{remark}


\section{Conditional Stochastic Collocation Method}\label{sec:cond_gPC}
With the conditioned random conductivity $\tilde{\kappa}(\mb{x},\w) = \exp[\tilde{Y}_{N_G}(\mb{x},\w)]$, Eq \eqref{eq:diffusion} has $N_G-N_m$ random dimensions and can be efficiently solved with the gPC numerical methods when $N_G-N_m$ is relatively small. Given that the computational cost of gPC-based methods exponentially increases with the number of random dimensions, the conditioning significantly reduces the computational cost of gPC. Here, we assume that  $N_G-N_m$ is small (i.e., the number of observations is comparable to $N_G$), and use the conditional stochastic collocation method \cite{XiuH_SISC05} as a surrogate for Eq \eqref{eq:diffusion}. 


\subsection{Generalized polynomial chaos}

Based on Eq (\ref{def:cond_rf_KL}), the conditional random field $\tilde{\kappa}(\bm{x},\omega)$ can be represented by KL expansion with $N_G-N_m$ terms as
\begin{equation}\label{eq:cond_kappa}
\begin{split}
\tilde{\kappa}(\mb{x},\w) &\sim \tilde{\kappa}(\mb{x},\bm{\xi}) \\ &=e^{\overline{Y}(\mb{x})+\sum_{n=1}^{N_G}\sqrt{\lambda_n}\e_n(\mb{x})\tilde{\mu}_n}e^{\sum_{i=1}^{N_G-N_m}\sqrt{\tilde{\lambda}_i}\tilde{\e}_i(\mb{x}){\xi}_i},
\end{split}
\end{equation}
and the ``conditional'' $u$, $\tilde{u}$, satisfies  
\begin{align}\label{eq:diffusion_cond}
\begin{split}
-\nabla \cdot (\tilde{\kappa}(\mb{x},\bm{\xi})\nabla \tilde{u}(\mb{x},\bm{\xi})) &= 0, \qquad \mb{x}\in D;\\
 \tilde{u}(\mb{x},\bm{\xi}) &= f(\mb{x}),\qquad \mb{x} \in \partial D_D;\\
 \vec{n}\cdot \tilde{\kappa}(\mb{x},\bm{\xi})\nabla  \tilde{u}(\mb{x},\bm{\xi})&= g(\mb{x}), \qquad \mb{x} \in \partial D_L.
\end{split}
 \end{align}

The gPC method for Eq (\ref{eq:diffusion_cond}) is based on an orthogonal polynomial approximation  of $\tilde{u}(\mb{x},\bm{\xi})$.
Let $\i = (i_1,\dots, i_N)\in \mathbf{N}_0^{n}$ be a multi-index with $|\i| = i_1+\dots+i_N$ and $P\geq 0$ be an integer. Then, the $P$th-degree gPC expansion of function $\tilde{u}(\mb{x},\bm{\xi})$ is defined as
\begin{equation}\label{eq:gPC_expan}
\tilde{u}(\mb{x},\bm{\xi}) \approx \sum_{|\i|=0}^P c_\i (\mb{x})     \Phi_{\i}(\bm{\xi}),
\end{equation}
where 
\[
c_\i (\mb{x}) = \E[\tilde{u}(\mb{x}, \bm{\xi})\Phi_\i(\bm{\xi})] = \int \tilde{u}(\mb{x} ,\bm{\xi})\Phi_\i(\bm{\xi})\rho(\bm{\xi})d\bm{\xi}
\]
are the coefficients of the expansion, $\Phi_\i(\bm{\xi})$ are the basis functions
\[
\Phi_\i(\bm{\xi}) = \phi_{i_1}(\xi_1)\dots\phi_{i_N}(\xi_N),\quad 0\leq|\i|\leq P,
\]
and $\bm{\xi} = (\xi_1,\dots,\xi_{N_G-N_m})$ is an $(N_G-N_m)$-dimensional random vector.
Here, $\phi_j(\xi_k)$ is the $j$th-degree one-dimensional orthogonal polynomial in $\xi_k$ direction satisfying
\[
\E[\phi_m(\xi_k)\phi_n(\xi_k)] = \delta_{m,n}, \quad 0\leq n,m\leq N.
\]
For identical independent distributed (i.i.d.) Gaussian random variables  $\{\xi_i\}_{i=1}^N$,  $\Phi_{\i}(\bm{\xi})$ are Hermite polynomials. The number of $\Phi_{\i}(\bm{\xi})$ is $( \begin{array}{c} N_G-N_m+P \\ P \end{array} )$ \cite{XiuK_SISC02,Xiu_CICP09}.
The classical approximation theory guarantees that the gPC approximation \eqref{eq:gPC_expan} converges to  $\tilde{u}(\mb{x},\bm{\xi})$ in $L^2$-norm as the degree $P$ increases when  $\tilde{u}(\mb{x},\bm{\xi})$ is square integrable with respect to the probability measure.

\subsection{Stochastic collocation method}
 In the stochastic collocation method, the expansion coefficients ${c}_{\i} (\mb{x})$ are approximated as
\begin{equation}\label{eq:coef_c}
{c}_{\i} (\mb{x}) \approx \tilde{c}_{\i} (\mb{x}) = \sum_{m = 1}^M \tilde{u}(\mb{x}, \bm{\xi}^{(m)})\Phi_{\i}(\bm{\xi}^{(m)})w^{(m)},
\end{equation}
where $\{\bm{\xi}^{(m)}\}_{m=1}^M$ is a set of quadrature points and $w^{(m)}, m = 1,\dots,M$ are the corresponding weights. 
For each collocation point $\bm{\xi}^{(i)}$, $u({\mb{x},\bm{\xi}^{(i)}})$ is obtained by solving \eqref{eq:diffusion_cond} with $\tilde{\kappa}(\mb{x},\bm{\xi})$ replaced by $\tilde{\kappa}(\mb{x},\bm{\xi}^{(i)})$.

Then, 
\begin{equation}\label{eq:uq_collo}
\tilde{u}(\mb{x},\bm{\xi})  \approx \tilde{u}_C(\mb{x},\bm{\xi}) = \sum_{|\i|=0}^P\tilde{c}_{\i}(\mb{x}) \Phi_\i(\bm{\xi})
\end{equation}
and the mean and variance of $\tilde{u}(\mb{x},\bm{\xi})$ can be approximated as
\begin{equation}\label{eq:fmean}
\E[\tilde{u}( \mb{x}, \bm{\xi})]\approx \E [\tilde{u}_C ( \mb{x}, \bm{\xi})] = \tilde{c}_0 (\mb{x})
\end{equation}
and 
\begin{equation}\label{eq:var}
\E[\tilde{u}(\mb{x}, \bm{\xi})-E[\tilde{u}(\mb{x},\bm{\xi})]]^2 \approx \E[\tilde{u}_C(\mb{x}, \bm{\xi})-E[\tilde{u}_C (\mb{x},\bm{\xi})]]^2 =\sum_{|\i|=1}^P \tilde{c}_{\i} (\mb{x})^2.
\end{equation}
There are various quadrature rules, including  tensor product quadrature rules for low-dimensional  $\bm{\xi}$ and the sparse grid methods for moderately dimensional $\bm{\xi}$ \cite{XiuH_SISC05,nobile2008sparse,Ma2009,jakemanG2013}. If the solution allows a low-dimensional representation, the compressed sensing can be used to decrease the number of collocation points  \cite{Yan_2012Sc,Yang_2013reweightedL1,Hampton_2015Cs}.  The convergence analysis of stochastic collocation methods is given in \cite{XiuH_SISC05,Xiu_CICP07}.

\begin{remark} In order to solve optimization problems, including (\ref{eq:opt_o}),(\ref{eq:opt_cond_rgPC}), and (\ref{eq:MAP}) (or to make the optimization problems well posed), additional measurements on solution (state) are required. 
Function $u(x,\omega)$ with random inputs $\omega$ and variance $\sigma^2_u(x)$ is more sensitive with respect to $\omega$ at position $x$ where the variance is larger.  Therefore, the optimal  location of $u(x,\omega)$ measurements for solving optimization problems should collocate with the local maxima of $\sigma^2_u(x)$.  In practice, only a few measurements are affordable; therefore, determining the optional measurement locations for the optimization problem is very important.
\end{remark}
 
\section{Bayesian estimation of $\tilde{\kappa}(\mb{x},\bm{\xi})$}\label{sec:inverse_problem}

We assume that in addition to $\{\hat{\kappa}(\mb{x}^j)\}_{j=1}^{N_m}$ measurements of $\hat{\kappa}$, we have $\{\hat{u}(\mb{x}^j)\}_{j=1}^{N_k}$ measurements of $\hat{u}$. With the conditional KL model (\ref{eq:cond_kappa}), we approximate $\hat{\kappa}(\bm{x})$ as  $\hat{\kappa}(\bm{x}) \approx \tilde{\kappa}(\bm{x},\bm{\xi}^*)$, where $\bm{\xi}^*$ is the $(N_G - N_m)$  dimensional deterministic vector. Then, the estimation of $\hat{\kappa}(\bm{x})$ reduces to the estimation of $\bm{\xi}^*$. With the surrogate model $\tilde{u}_C(\mb{x},\bm{\xi})$, we estimate $\bm{\xi}^*$ by solving the following minimization problem:

\begin{equation}\label{eq:opt_cond_rgPC}
\bm{\xi}^* = \textrm{argmin}_{\bm{\xi}\in \mathbb{R}^{N_G-N_m}}\sum_{j=1}^{N_k}|\hat{u}(\mb{x}^j)-\tilde{u}_C(\mb{x}^j,\bm{\xi})|^2 + \lambda \|\bm{\xi}-\bm{\xi}^o\|^2,
\end{equation}
where the regularization term $ \|\bm{\xi}-\bm{\xi}^o\|^2$  was proposed in \cite{stuart15c} and $\bm{\xi}^o\in \mathbb{R}^{N_G-N_m}$ and $\lambda$ are the regularization parameters.  Solving this optimization problem using standard minimization methods could be computationally expensive for large $N_G-N_m$.  Therefore, we employ a Bayesian statistical approach to solve the optimization problem and provide the correspondence between the regularization parameters and the prior statistics of $\bm{\xi}$.

In Eq \eqref{eq:opt_cond_rgPC}, we assume that the errors $\{\hat{u}(\mb{x}^j)-\tilde{u}_C(\mb{x}^{j},\bm{\xi})\}_{j=1}^{N_k}$ are i.i.d. random variables, that is,
\begin{equation*}
\hat{u}(\mb{x}^{j})= \tilde{u}_C(\mb{x}^{j},\bm{\xi}) + {\delta_j}, \quad {j=1,\dots,N_k},
\end{equation*}
where $\delta_j$ $j=1,\dots, N_k$ is i.i.d. random variable with a certain density $p_{\delta}$. Following a common practice in parameter estimation, we assume that $p_{\delta}$ is a Gaussian density with mean $0$ and standard deviation $\sigma_{\delta}\ll 1$. Then, the likelihood takes the form 
\[
L(\bm{\xi}) =  \prod _{j=1}^{N_k}p_{\mb{\delta}}\left(\delta_j\right) = \prod _{j=1}^{N_k}p_{\delta}(\hat{u}(\mb{x}^{j})-\tilde{u}_C(\mb{x}^{j},\bm{\xi})).
\]
With $p_{\bm{\xi}}$ denoting the prior probability density for $\bm{\xi}$, we use Bayes' rule to obtain a posterior probability density for $\bm{\xi}$ given the observations $\{\hat{u}(x^{(j)})\}_{j=1}^{N_k}$:
\begin{equation}\label{def:para_pos}
p(\bm{\xi}|\hat{u}(\bm{x}^{j}),j=1,\dots,N_k) \propto p_{\bm{\xi}}\prod _{j=1}^{N_k}p_{\delta}(\hat{u}(\mb{x}^{j})-\tilde{u}_C(\mb{x}^{j},\bm{\xi})).
\end{equation}
The (prior) density $p_{\bm{\xi}}$ of $\bm{\xi}$ is usually chosen to be Gaussian, that is, $\bm{\xi}\sim \mathcal{N}(\bm{\xi}^{o},\theta \mb{I})$ with the hyper-parameter $\theta$ controlling the prior variance and $\bm{\xi}^o$ being the prior mean. Usually, the prior mean of $\bm{\xi}$ is set to $\bm{0}$. The posterior density of $\bm{\xi}$ with hyper-parameter is rewritten as,
\begin{equation}\label{eq:posterier_full}
\begin{split}
& p(\bm{\xi},\theta|\hat{u}(\bm{x}^{(j)},j=1,\dots,N_k)) \propto p_{\bm{\xi}}\prod _{j=1}^{N_k}p_{\delta}(\hat{u}(\mb{x}^{j})-\tilde{u}_C(\mb{x}^{j},\bm{\xi}))\\
& \propto \exp(-\frac{\sum_{j=1}^{N_k}|\hat{u}(\mb{x}^{j})-u_P(\mb{x}^{j},\bm{\xi})|^2}{2\sigma_{\delta}})\exp(-\frac{\|\bm{\xi}-\bm{\xi}^o\|^2}{2\theta}).
\end{split}
\end{equation}
Then, the maximum a posteriori (MAP) estimator of this posterior distribution is 
\begin{equation}\label{eq:MAP}
\textrm{argmin}_{\bm{\xi}\in \mathbb{R}^{N_G-N_m}}\frac{\sum_{j=1}^{N_k}|\hat{u}(\mb{x}^i)-\tilde{u}_C(\mb{x}^j,\bm{\xi})|^2}{2\sigma_{\delta}} + \frac{ \|\bm{\xi}-\bm{\xi}^o\|^2}{2\theta}.
\end{equation}
Note that \eqref{eq:MAP} is a special case of the minimization problem \eqref{eq:opt_cond_rgPC} with $\lambda = \frac{\sigma_{\delta}}{\theta}$.

 To generate posterior distributed sampling, we employ the widely used Differential Evolution Adaptive Metropolis (DREAM) MCMC toolbox developed by Vrugt\cite{VRUGT2016}. Given the prior $\bm{\xi}_0 = \bm{0}$, and $\theta = 1$ such that $\bm{\xi}$ is consistent to the form of Gaussian process \eqref{def:cond_rf_KL}, we generate the samples of $\bm{\xi}$ for the MCMC. For $\sigma_{\delta}$, we assume the errors are small, that is, we choose $\sigma_{\delta}$ to be about $1e-3$. Whenever a sample $\bm{\xi}$ is generated, the surrogate  $\tilde{u}_C(\bm{x}^i,\bm{\xi})$ is evaluated and the errors $|\hat{u}(\mb{x}^i)-\tilde{u}_C(\mb{x}^i,\bm{\xi})|^2$, $i=1,N_k$ in the  Metropolis Hastings algorithm are computed. Then, the DREAM MCMC toolbox is employed to generate samples of posterior distribution and the MAP probability estimate of this posterior distribution is used as an approximate solution of the optimization problem \eqref{eq:opt_cond_rgPC}. The numerical examples in Section \ref{sec:num}  show the accuracy of this approximation.
 
\subsection{Determining measurement locations of the state $u(\mb{x})$}\label{u_sampling}
The choice of $u(\mb{x})$ measurements $\{\hat{u}(\mb{x}^j)\}_{j=1}^{N_k}$ can impact the posterior distributions obtained through MCMC. Because $\sigma^2_u(\mb{x})$, the variance of $u(\mb{x},\bm{\xi})$,  describes the variability of $u(\mb{x},\bm{\xi})$ with respect to the parameters $\bm{\xi}$, we propose the following strategy:
\begin{itemize}
\item[1] Approximate $\sigma^2_u(\mb{x})$ by $\tilde{\sigma}^2_u(\mb{x})  = \sum_{|\mb{i}|=1}^P  \tilde{c}^2_{\mb{i}}(\mb{x})$.
\item[2] Find all local maxima and saddle points on the $\tilde{\sigma}^2_u({\mb{x}})$ denoted by $S:=\{ \tilde{\sigma}^2_u({\mb{x_*^j}})\}_{j=1}^{N_{\ell}}$ in descending order. 
\item[3] If $N_{\ell}\geq N_k$, we choose the first $N_k$ locations from ordered set $S$.
\item[4] If ${N_{\ell}}<N_k$, divide the physical domain $D$ in $N_k$ equal blocks.
For blocks without saddle points from $S$, find  $\sigma^2_u(\mb{x})$ maxima in this block, arrange them in descending order, and choose the first $N_{k}-N_{\ell}$ locations to add to $S$. 
\end{itemize}
 The numerical examples show that this strategy outperforms  equal or random distribution of $u(\bm{x})$ measurements.  


\section{Numerical examples}\label{sec:num}
We implement the proposed parameter estimation approach for one- and two-dimensional steady-state diffusion equations. We assume that the ``reference'' diffusion coefficient is an instance of the known Gaussian process that can be accurately represented with the the $N_G$-dimensional KL expansion \eqref{def:truncated_KL}. As a result, the reference $\kappa$ filed lies in the space of the conditional KL expansion (\ref{eq:cond_kappa}). 


 
\subsection{One-dimensional example}
Consider the following one-dimensional steady-state diffusion equation with the unknown coefficient $\hat{\kappa}(x)$ modeled as the random field $\kappa(x,\omega)$:
\begin{equation}\label{eq:diffu_1d}
\begin{split}
\frac{\partial }{\partial x}(\kappa(x,\omega)\frac{\partial}{\partial x}u(x,\omega)) &= 0, \quad x\in(0,1)\\
u(0,\omega) = u_{l},\quad u(1,\omega) &= u_r, 
\end{split}
\end{equation}
where $u_l= 0$ and  $u_r = 2$.
We assume that $\kappa(x,\omega)$ has lognormal distribution (i.e, $g(x,\omega) = \log(\kappa(x,\omega))$  is Gaussian) with mean $\mu_k = 5.0$,  variance $\sigma_k = 2.5$ and the covariance function 

\begin{equation}
C(x_1,x_2) = e^{-\frac{(x_1-x_2)^2}{L^2}}, \quad x_1,x_2\in [0,1],
\end{equation} 
where the correlation length is $L = 0.05$ (i.e., $L$ is 20 times smaller than the domain size). 

We find that the finite KL expansion with $N_G=25$ terms,
\begin{equation}\label{eq:g_KL}
g(x,\omega) = \mu_g+ \sigma_g\sum_{i=1}^{N_G=25}\sqrt{\lambda_{i}}\epsilon_i(x)\xi_i,
\end{equation}
captures $95\%$ of the spectrum of this field covariance function. Here,  the mean $\mu_g$ and standard deviation $\sigma_g$ of $g(x,\omega)$ are given in terms of  $\mu_g$ and $\sigma_g$ as
\begin{equation}\label{eq:mu_sigma}
\begin{split}
\mu_g &= \log\mu_k-\frac{1}{2}\log{(\frac{\sigma_k^2}{\mu_k^2}+1)},\\
\sigma_g &=\sqrt{\log{(\frac{\sigma_k^2}{\mu_k^2}+1)}}.
\end{split}
\end{equation}

Then, we compute the reference conductivity field ${\hat{\kappa}}(x)$ as a realization of Eq (\ref{eq:g_KL}) and solve Eq (\ref{eq:diffu_1d})  with the finite element method to obtain the reference solution, $\hat{u}(x)$.  
Next, we select $N_\kappa$ values of $\hat{\kappa}(x)$ and $N_u$ values of $\hat{u}(x)$ as measurements of these fields and reconstruct  $\hat{\kappa}(x)$ as $\tilde{\kappa}(x,\bm{\xi}^*)$ by solving the minimization problem (\ref{eq:MAP}). The conditioned surrogate model $\tilde{u}_C$ in (\ref{eq:MAP}) is constructed by solving Eq (\ref{eq:diffu_1d}) with the finite element method at collocation points. 
In the finite element solutions, the domain is discretized with $256$ equal size elements. We study the relative error of the inferred conductivity $\tilde{\kappa}(x,\bm{\xi}^*)$,
\begin{equation}\label{relative_er}
{\varepsilon(x) = \frac{|\tilde{\kappa}(x,\bm{\xi}^*)-\hat{\kappa}(x)|}{\hat{\kappa}(x)}.}
\end{equation}
as a function of $N_\kappa$, $N_u$, and the sampling strategy. 

We consider three  $\hat{\kappa}(x)$ sampling strategies: Case 1, the observation locations of $\hat{\kappa}(x)$  are chosen randomly; Case 2,  the observation locations are equally distributed; and Case 3, the observation locations coincide with locations of local minima and maxima. The referenced field $\kappa(x)$ and the locations of observations for each case are presented in Figure \ref{fig:1d_k_obs_dist}a.
\begin{figure}[htbp]
\centerline{
\psfig{file = 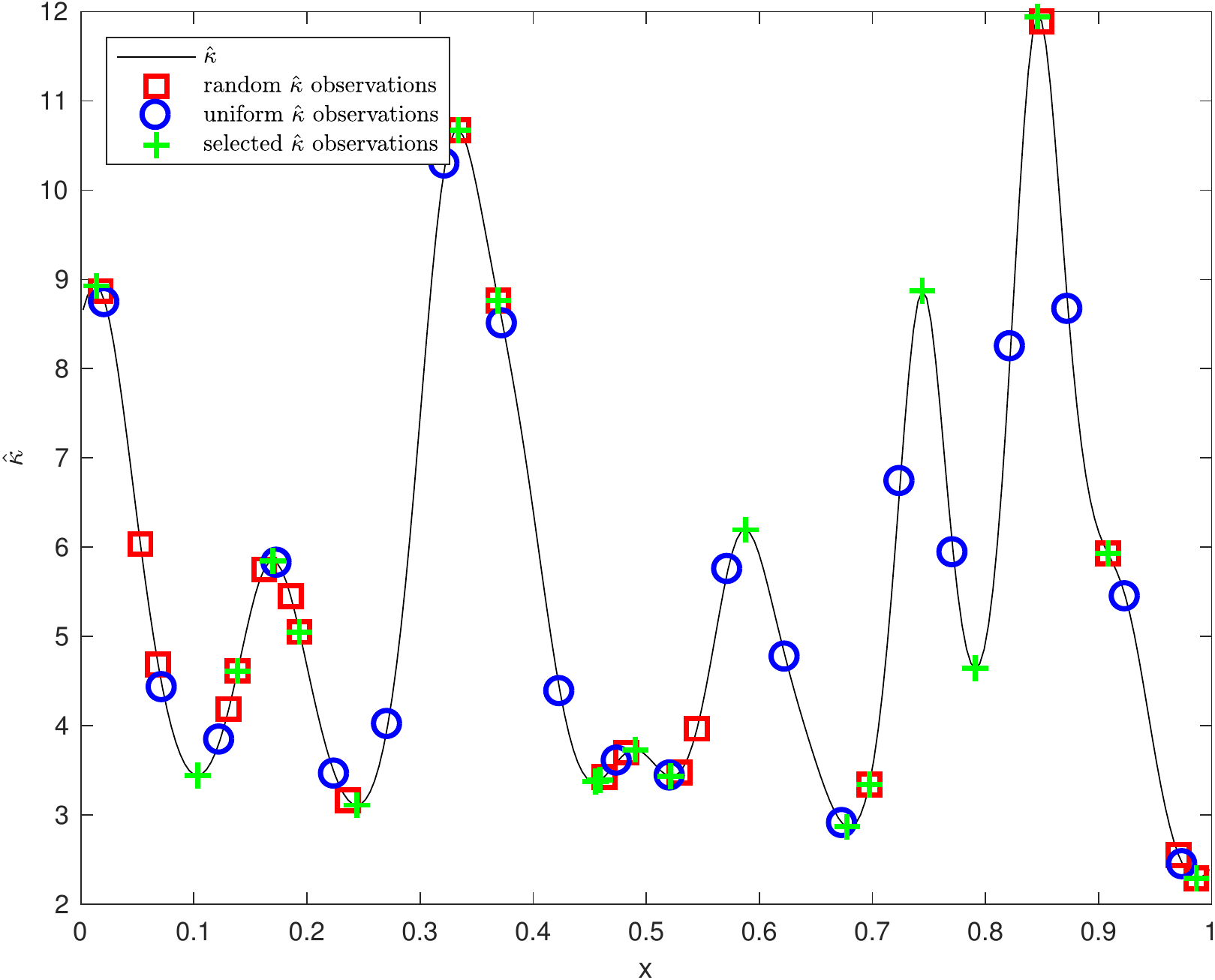, width = 7cm}
\psfig{file = 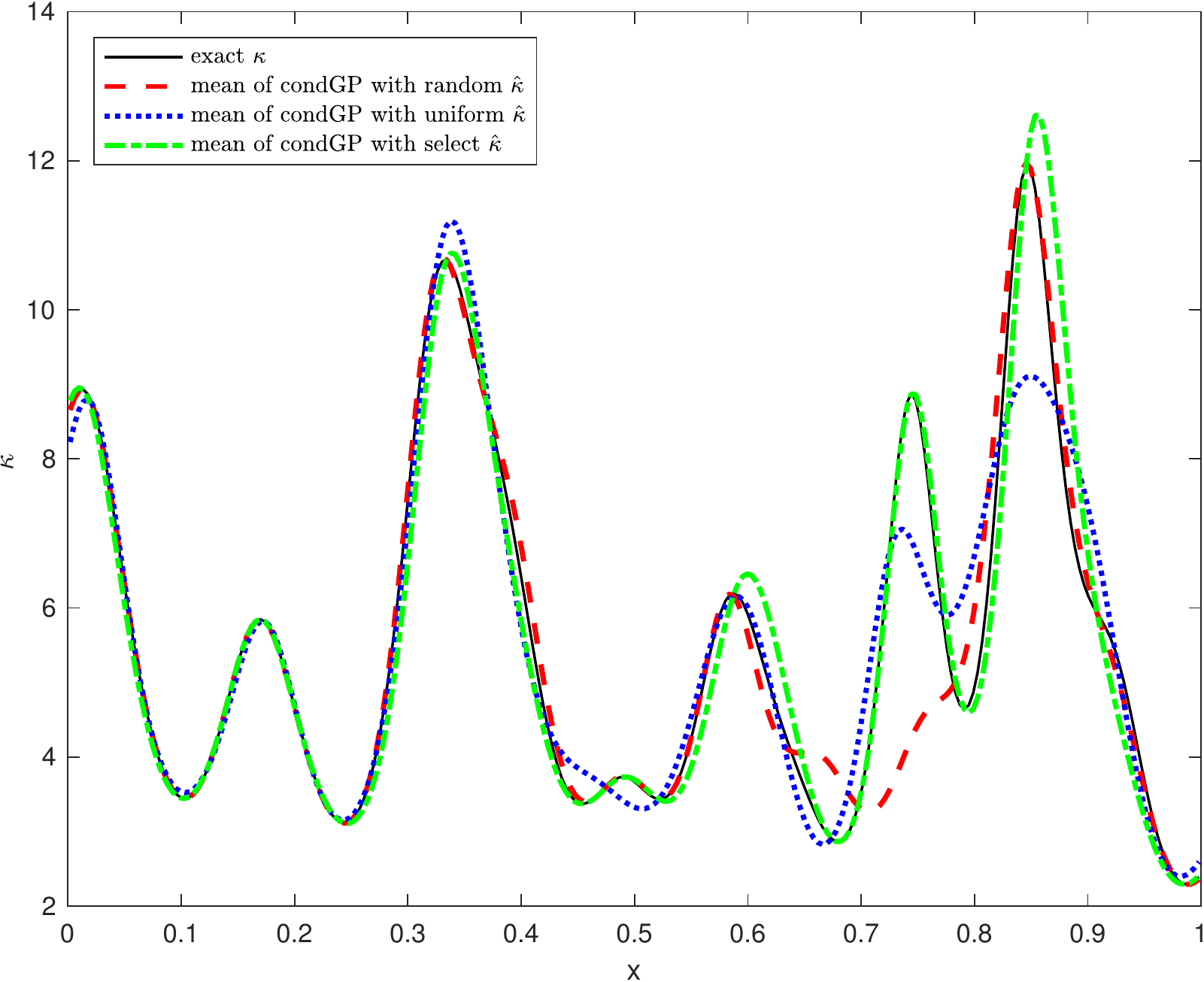, width = 7cm}
}
\caption{Reference field $\hat{\kappa}(x)$ and the observation locations for Cases 1, 2 and 3 (left) and the mean of the conditional GP field constructed in Case 1, 2, 3 (right).}
\label{fig:1d_k_obs_dist}
\end{figure}

We also study the effect of $\hat{u}$ measurement locations on the relative error $\varepsilon(x)$. As suggested in Remark 3, $u(x,\bm{\xi})$ is more sensitive with respect to the random inputs $\bm{\xi}$ where its variance is  large. Therefore, it reasonable to assume that the best candidate locations of $\hat{u}$ measurements are those where $\sigma^2_u (x)$ has local maxima. The $u$ variance $\sigma^2_u (x)$ can be evaluated with the conditional gPC method using Eq \eqref{eq:var}. 
On the other hand, the MCMC optimization algorithm requires $u$ measurements to be uncorrelated. 
Therefore, we compare the accuracy of the $u$ sampling strategy outlined in Section \ref{u_sampling} (based on $\sigma^2_u (x)$ local maxima) with uniform and random sampling strategies. 

\subsubsection{Case 1: Random-chosen locations of the conductivity measurements}
Here we assume that no expert knowledge about the conductivity field is available and the 20 sampling locations of $k$ are chosen randomly, as shown in Figure \ref{fig:1d_k_obs_dist} by red square markers. It can be seen that $\hat{\kappa}$ oscillates significantly between  the sampling points and cannot be accurately learned via regression only, as shown in Figure \ref{fig:1d_k_obs_dist}b.  


To chose the $u$ measurement locations, we compute the variance of $u(x,\bm{\xi})$ using \eqref{eq:var} (shown in Figure \ref{fig:inverse_1d_randomk_uvar}) and then perform the following three tests: 
\begin{itemize}
\item[(a)] We randomly choose six locations of the reference solution $\hat{u}$ for the optimization step as shown in Figure \ref{fig:inverse_1d_randomk_uvar} by red ``$\square$'' markers. 
\item[(b)] We choose six equally distributed locations of $\hat{u}$, as shown in Figure \ref{fig:inverse_1d_randomk_uvar} by blue ``$\circ$'' markers. 
\item[(c)] We choose six locations using algorithm in Section \ref{u_sampling} as shown in Figure \ref{fig:inverse_1d_randomk_uvar} by green ``$+$'' markers. 
\end{itemize} 
Figure \ref{fig:inverse_1d_randomk_diff_u}(a) shows the reference and estimated conductivity field. The associated $\varepsilon(x)$ errors  are presented in  Figure \ref{fig:inverse_1d_randomk_diff_u}(b).
We can see that in all three tests, the errors are largest between the two cross markers where there are no $\hat{\kappa}$ measurements. On the other hand,  optimally selected locations of $\hat{u}$ in test (c) improve parameter estimation, i.e., reduce the  relative errors $\|\varepsilon\|_{L^{\infty}}$ from more than $60\%$  (for randomly or equally distributed $\hat{u}$ measurements) to less than $14\%$. 

\begin{figure}[htbp]
\centerline{
\psfig{file = 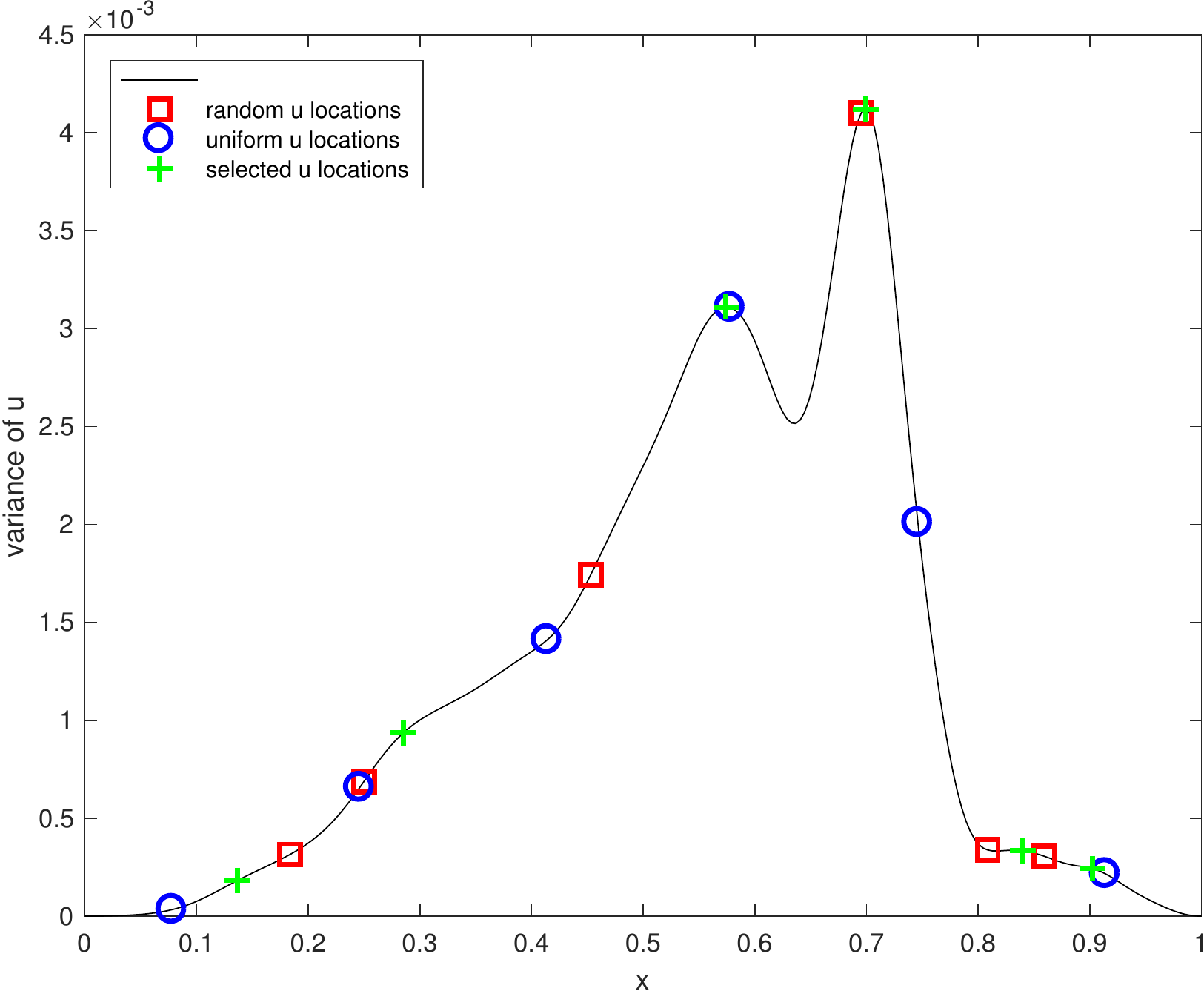, width = 7cm}
}
\caption{Example 1 Case 1 (randomly distributed $\kappa(x)$ measurements). 
Symbols denote locations of $u$ measurements selected randomly, uniformly, and based on $\sigma^2_u(x)$.  Black line denotes $\sigma^2_u(x)$.}
\label{fig:inverse_1d_randomk_uvar}
\end{figure}

\begin{figure}[htbp]
\centerline{
\psfig{file = 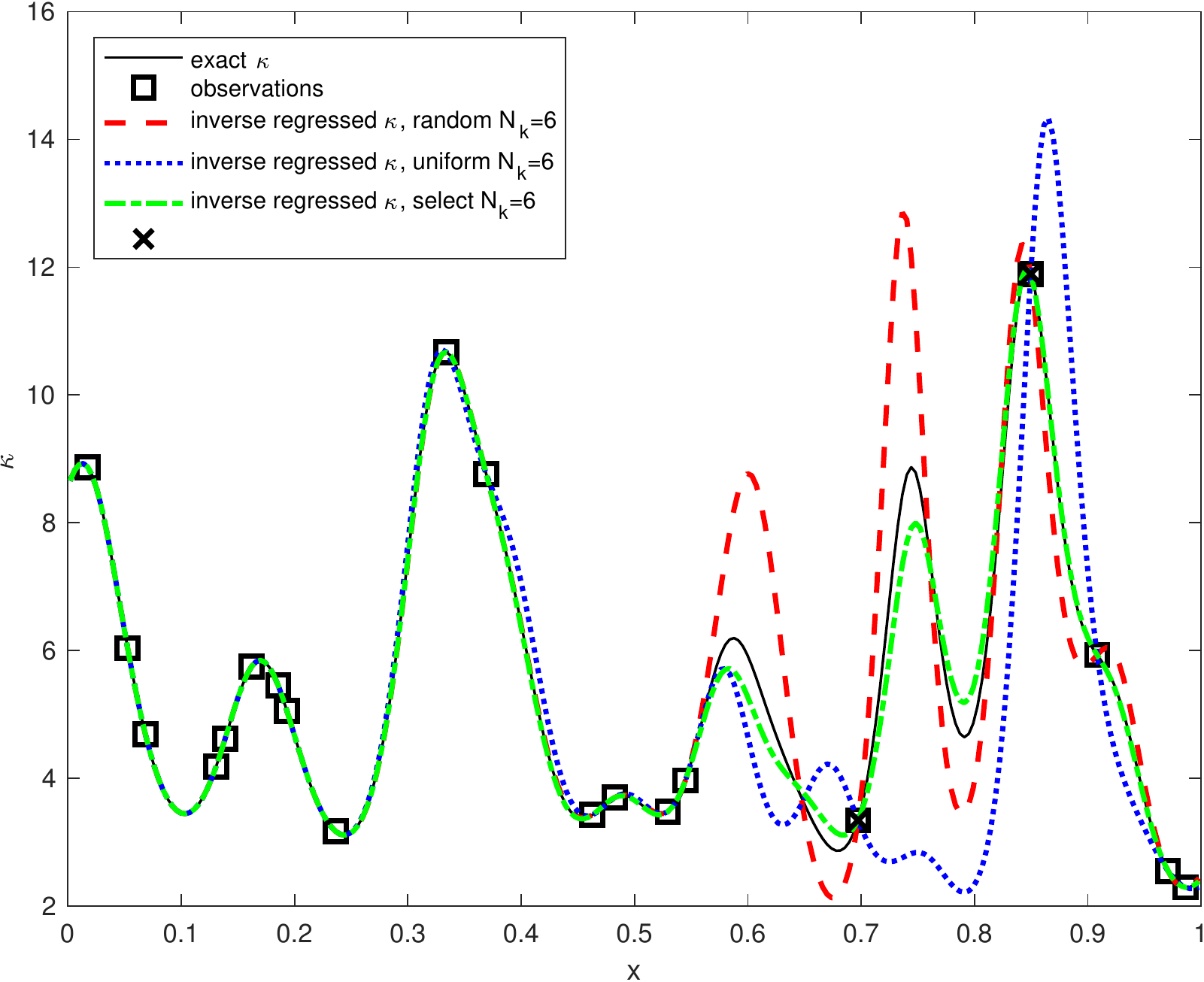, width = 7cm}
\psfig{file = 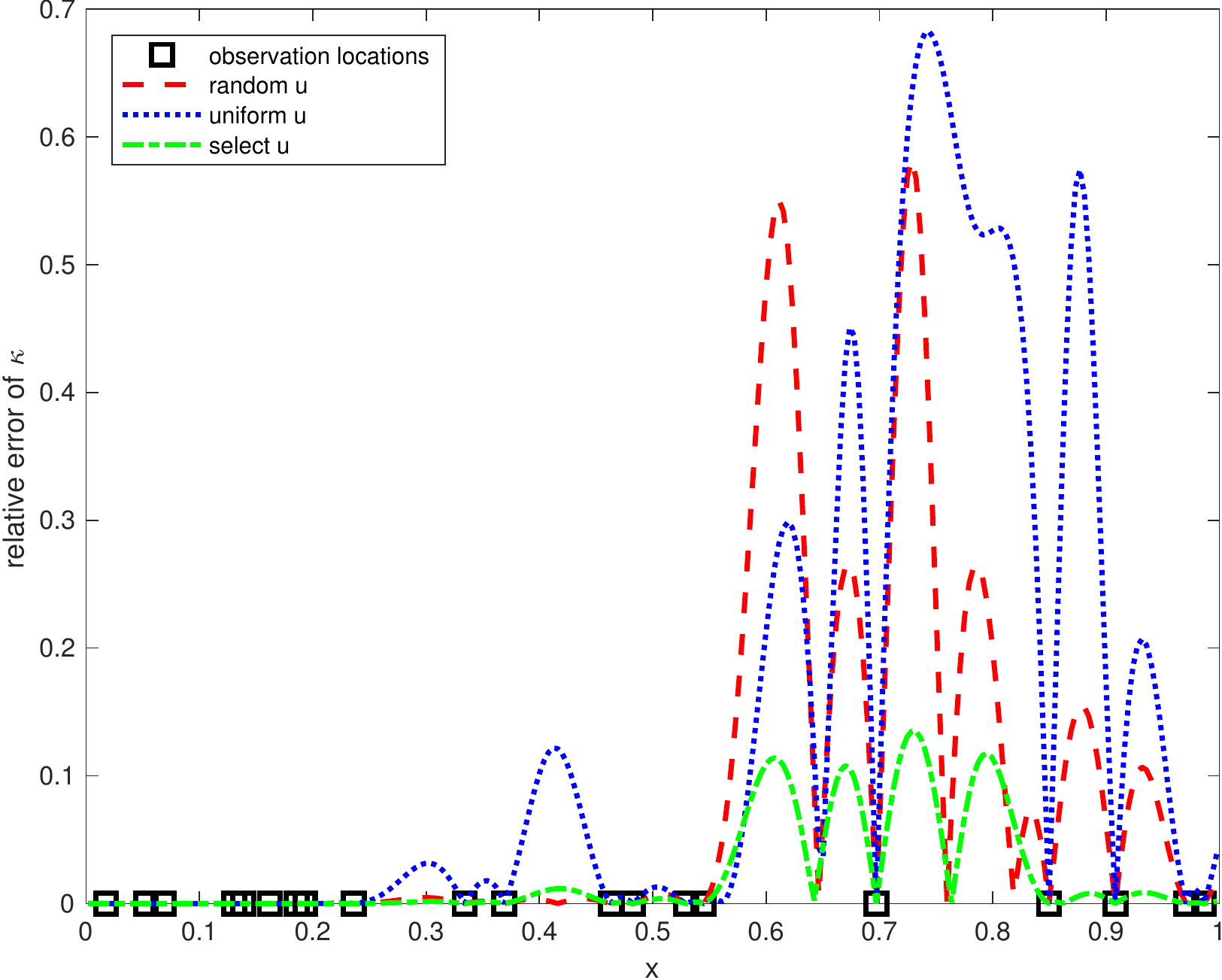, width = 7cm}
}
\caption{Example 1 Case 1. Estimated $\hat{\kappa}$ using  $u$ measurement locations obtained with three $u$ sampling strategies. Locations of $k$ measurements are shown with open square markers.}
\label{fig:inverse_1d_randomk_diff_u}
\end{figure}

\subsubsection{Case 2:  Observations of the conductivity coefficient are equally distributed} 
Next, we consider $\kappa$ observation locations equally distributed, as shown in Figure \ref{fig:1d_k_obs_dist}a by blue ``$\circ$'' markers. We test the same three strategies for selecting $u$ measurement locations as in Case 1, including randomly and uniformly distributed locations and locations chosen based on the local maxima of $\sigma^2_u(x)$. The resulting distributions of the $u$ measurements locations as well as $\sigma^2_u(x)$ are shown in Figure \ref{fig:inverse_1d_uniformk_uvar}. 
\begin{itemize}
\item[(a)] We randomly choose six locations for the measurements of $\hat{u}$, as shown in Figure \ref{fig:inverse_1d_uniformk_uvar} by red $\square$ markers. 
\item[(b)] We choose six equally distributed locations, 
 as denoted in Figure \ref{fig:inverse_1d_uniformk_uvar} by blue $\circ$ markers. 
\item[(c)] We  choose six locations  using algorithm in Section \ref{u_sampling}, as shown in Figure \ref{fig:inverse_1d_uniformk_uvar} by green $+$ markers. 
\end{itemize} 

\begin{figure}[htbp]
\centerline{
\psfig{file = 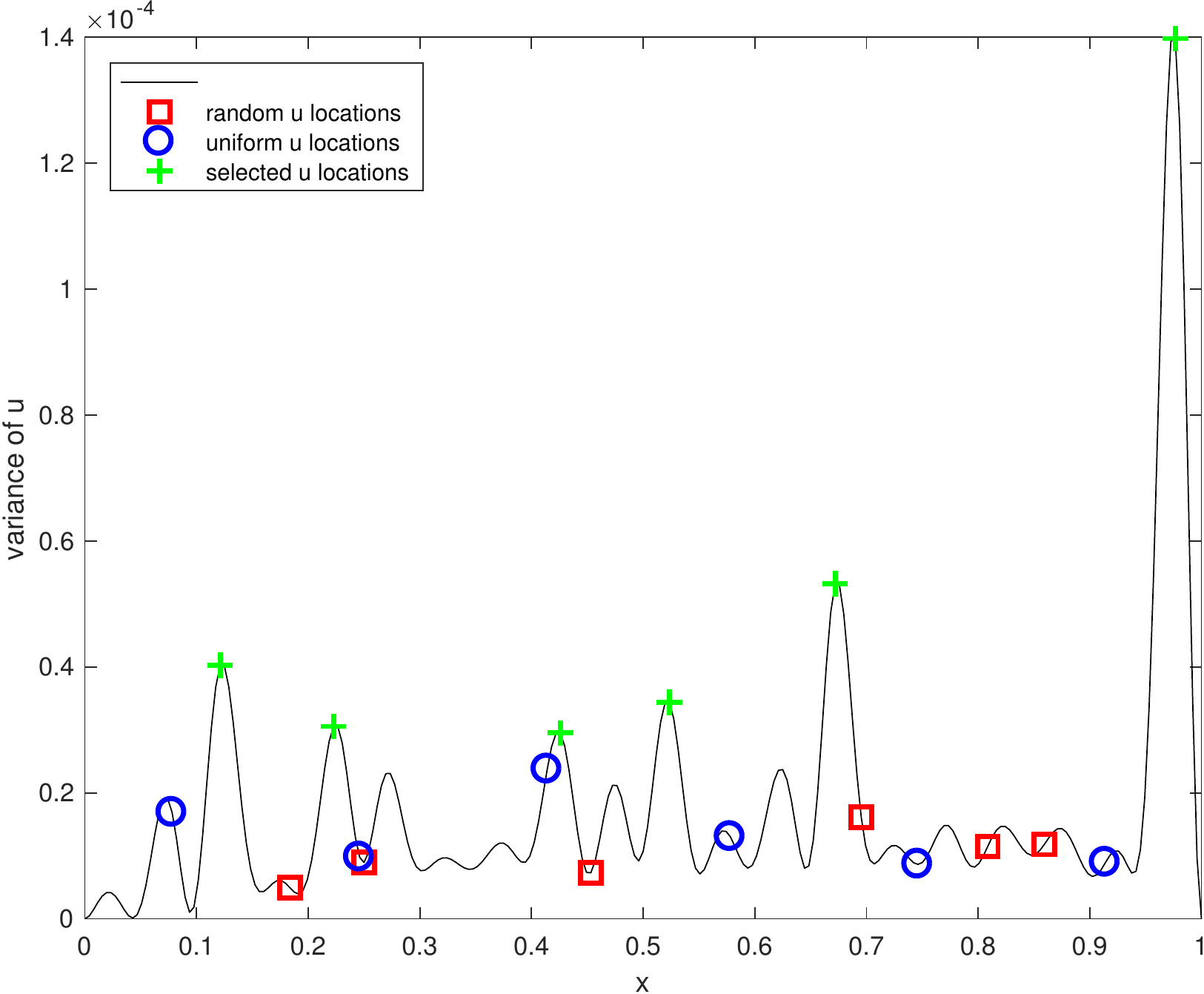, width = 7cm}
}
\caption{Example 1 Case 2 (uniformly distributed $\kappa(x)$ measurements). 
Symbols denote locations of $u$ measurements selected randomly, uniformly, and based on $\sigma^2_u(x)$.  Black line denotes $\sigma^2_u(x)$.}
\label{fig:inverse_1d_uniformk_uvar}
\end{figure}

\begin{figure}[htbp]
\centerline{
\psfig{file = 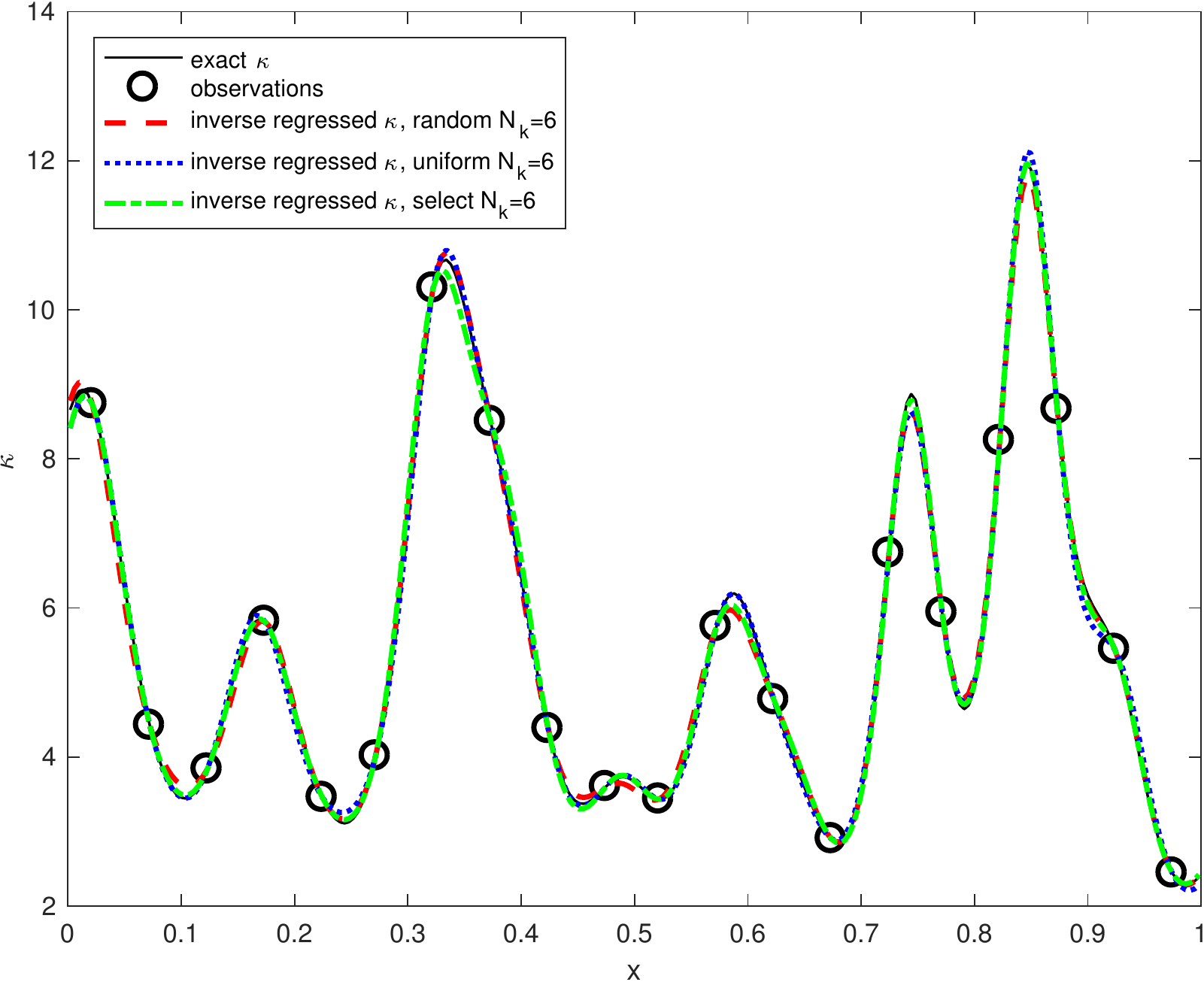, width = 7cm}
\psfig{file = 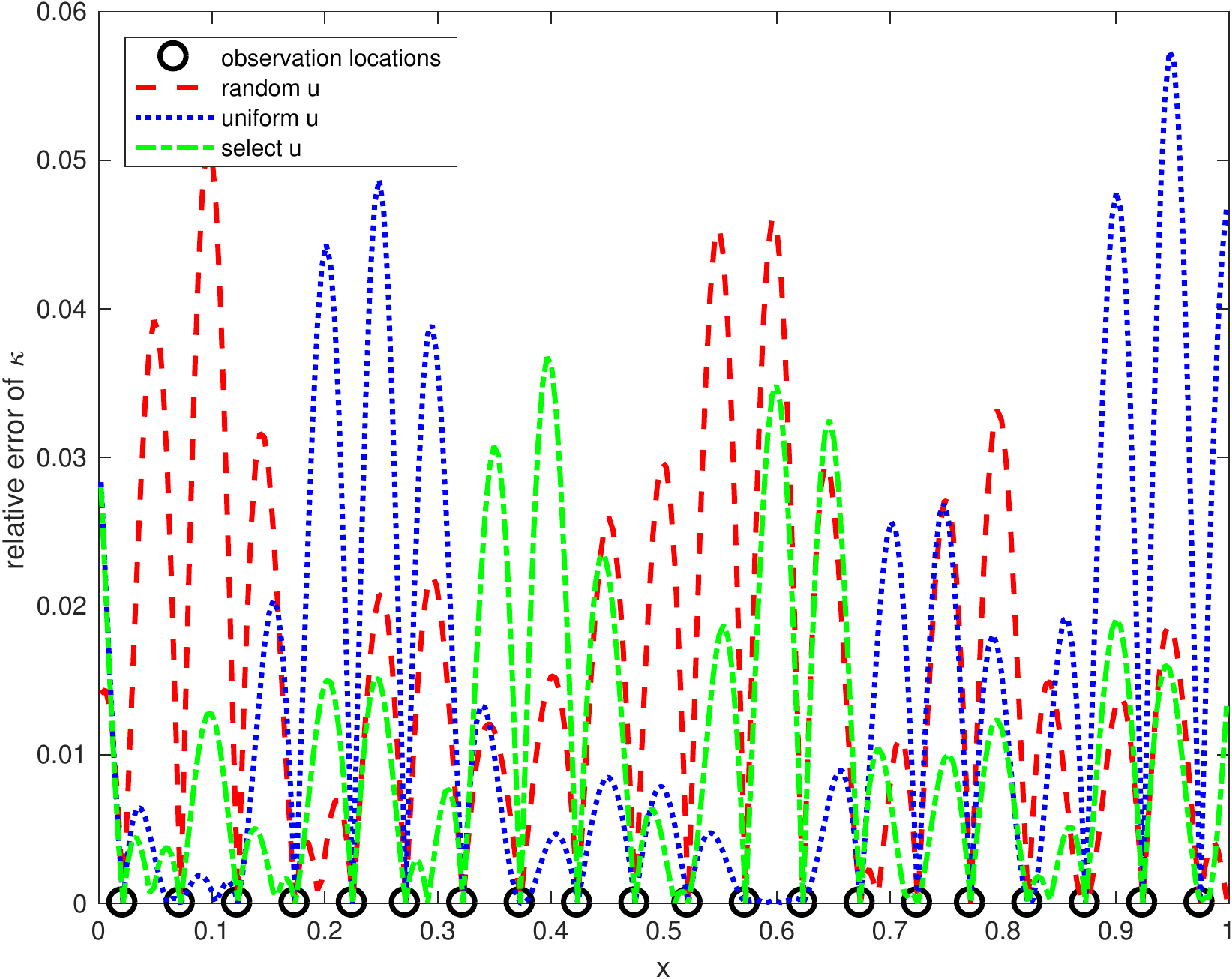, width = 7cm}
}
\caption{Example 1 Case 2. Regressed $\hat{\kappa}$ with uniform observations using  different state $u$ measurements.}
\label{fig:inverse_1d_uniformk_diff_u}
\end{figure}

Figure \ref{fig:inverse_1d_uniformk_diff_u}a shows the reference and estimated conductivity fields  and Figure \ref{fig:inverse_1d_uniformk_diff_u}b presents the associated relative errors $\varepsilon(x)$.
Here, the relative errors in $\hat{\kappa}(x)$ are less than $5\%$ for all considered choices of $u$ measurement locations.
The comparison of relative errors in Figure \ref{fig:inverse_1d_uniformk_diff_u}(right) indicates that the selection of the $\hat{u}$ measurement locations based on $\sigma^2_u(x)$ has the smallest $\|\varepsilon(x)\|_{\infty}$ error. 

\subsubsection{Case 3: Observations of $\kappa$ are colocated with local maxima and minima of $\hat{k}(x)$}
Here, we assume that the locations of local maxima and minima of $\hat{\kappa}(x)$ are known (e.g., based on expert knowledge). Then, we take observations at  all $14$ local minima and local maxima and one inflection point of $\hat{\kappa}(x)$. Another five observations of $\hat{\kappa}(x)$ are taken at random locations. The selected observations are shown by green ``$+$'' markers in Figure \ref{fig:1d_k_obs_dist}a. 

\begin{figure}[htbp]
\centerline{
\psfig{file = 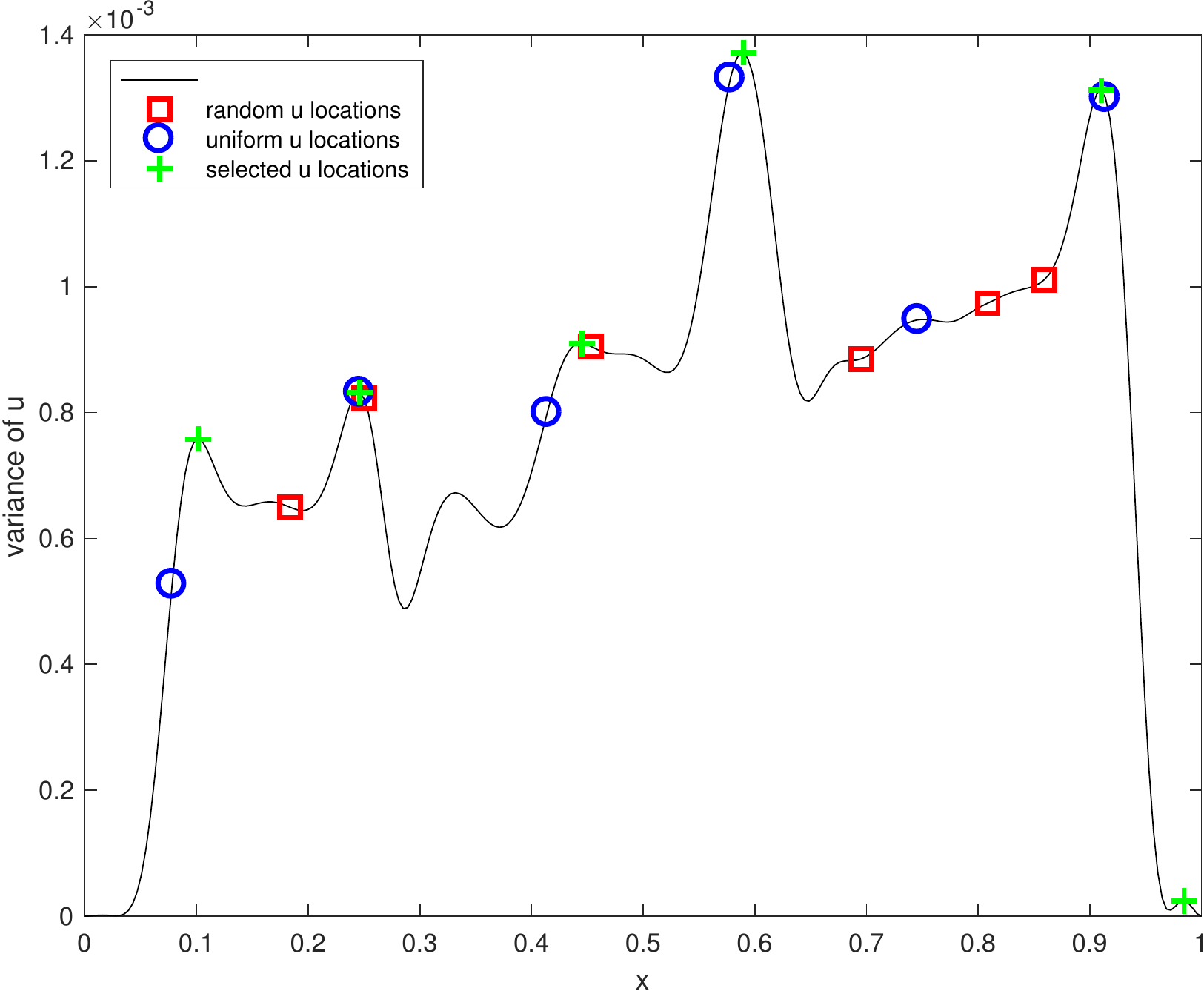, width = 7cm}
}
\caption{Example 1 Case 3 ($\kappa$ measurements collocated with local maxima and minima of $\kappa(x)$). 
Symbols denote locations of $u$ measurements selected randomly, equally, and based on $\sigma^2_u(x)$.  Black line denotes $\sigma^2_u(x)$.}
\label{fig:inverse_1d_selectk_uvar}
\end{figure}

\begin{figure}[htbp]
\centerline{
\psfig{file = 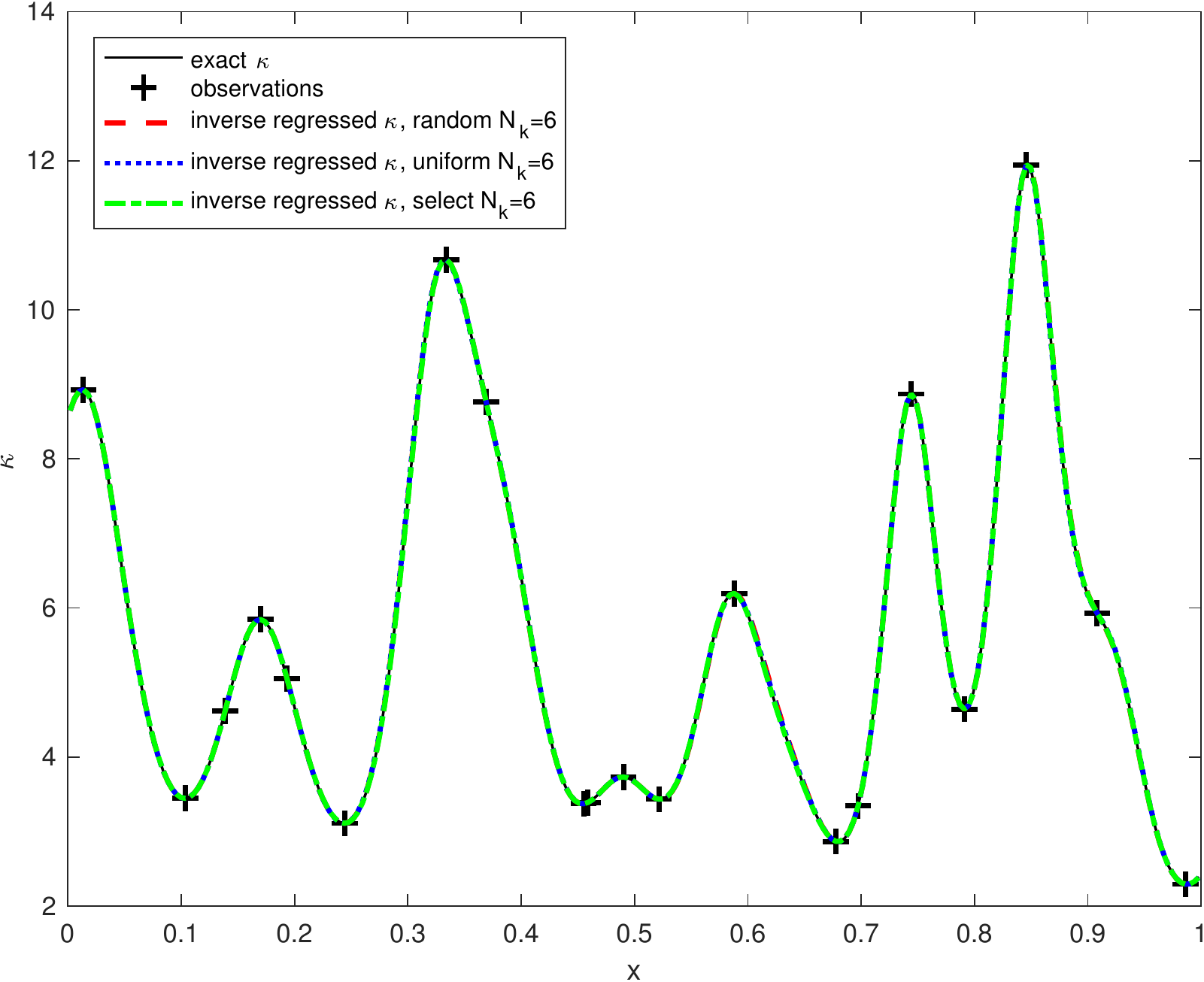, width = 7cm}
\psfig{file = 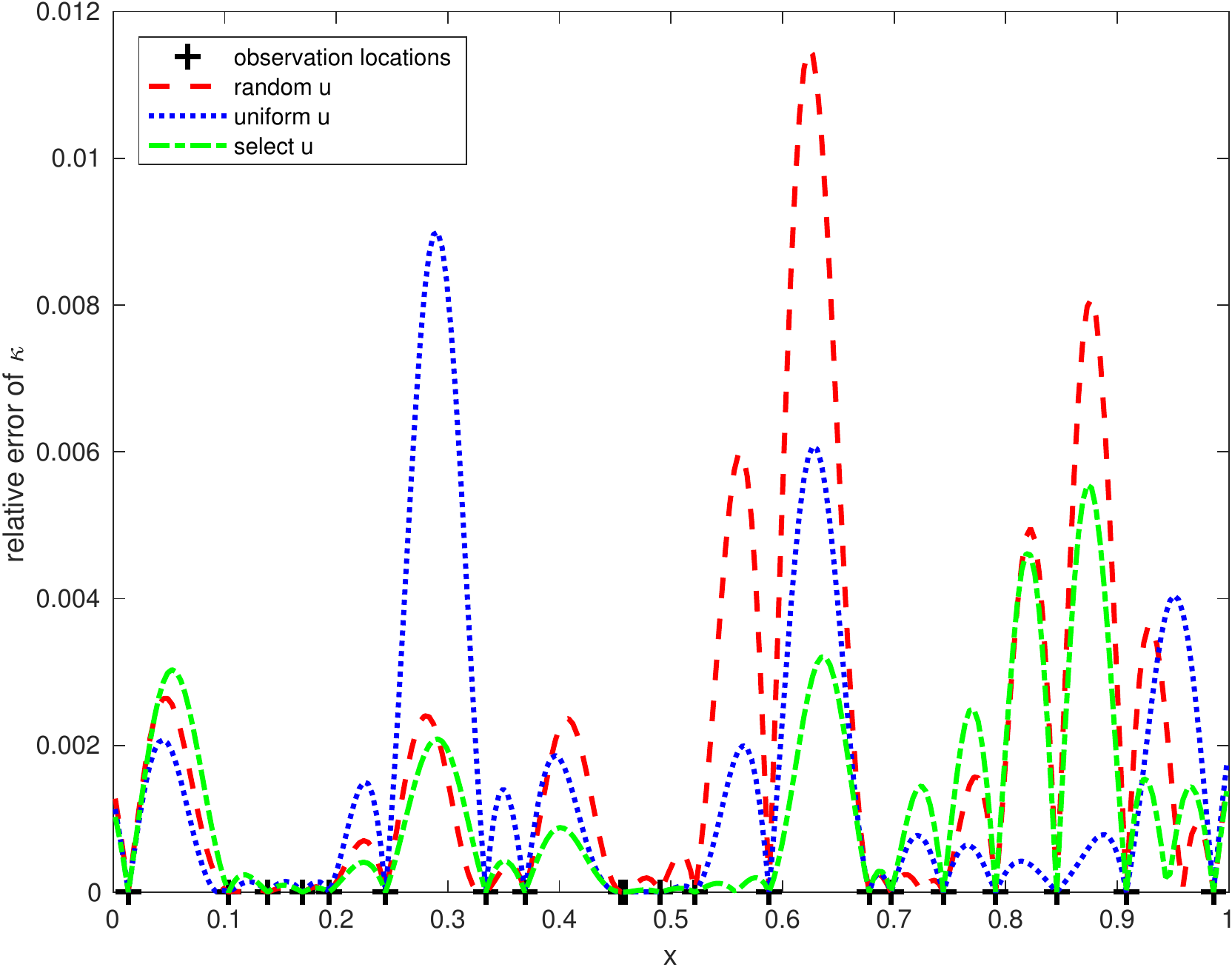, width = 7cm}
}
\caption{Example 1 Case 3. Regressed $\hat{\kappa}$ with select observations using  different state $u$ measurements.}
\label{fig:inverse_1d_selectk_diff_u}
\end{figure}

We test the same three strategies for selecting the measurement locations of $\hat{u}(x,\xi)$ as in Cases 1 and 2. Figure 
 \ref{fig:inverse_1d_selectk_uvar} shows the conditional variance of $u$ and $u$ measurement locations selected randomly, uniformly, and according to the conditional variance of $u$.
 Figure \ref{fig:inverse_1d_selectk_diff_u} shows the reference and estimated $\kappa$ as well as the corresponding errors.  
For all considered $u$ sampling strategies, $\varepsilon(x)$ are less than $1.2\%$.
Among them, the $u$ measurement locations selected according to $\sigma^2_u(x)$ give the smallest $\|\varepsilon(x)\|_{\infty}$, which is less than $0.6\%$.

\subsection{Two-dimensional example with large correlation length of $\kappa$ (smooth $\kappa$ field)}
Here, we consider the two-dimensional diffusion equation with the partially observed coefficient $\kappa$ modeled as a random field $\kappa(\mb{x},\omega)$,
\begin{equation}\label{eq:diffusion_2d}
\nabla \cdot (\kappa(\mb{x},\bm{\xi})\nabla u(\mb{x},\bm{\xi})) = 0, \qquad \mb{x}=(x_1,x_2)\in D,\\
\end{equation}
subject to the boundary conditions
\begin{equation}\label{eq:diff_2d_bd}
\begin{split}
&u(0,x_2,\bm{\xi}) = 50,\quad u(240,x_2,\bm{\xi}) = 25,\\
&-\mb{n}\cdot (\kappa(\mb{x},\bm{\xi})\nabla u)|_{(x_1,0)} = \mb{n}\cdot (\kappa(\mb{x},\bm{\xi})\nabla u)|_{(x_1,60)} = 0, 
\end{split}
\end{equation}
where $D = [0,240]\times[0,60]$ and $\mb{n} = (0,1)$.

We assume that $\kappa(\mb{x},\bm{\xi})$ is a lognormal field with mean $\mu_\kappa = 5$, standard deviation $\sigma_\kappa =2.5$, and $g(\mb{x},\bm{\xi}) = \log \kappa(x,\bm{\xi})$ is given by the truncated KL expansion:
$$g(\mb{x},\bm{\xi}) = \mu_g+ \sigma_g\sum_{i=1}^{25}\sqrt{\lambda_{i}}\epsilon_i(\mb{x})\xi_i,$$
where $\mu_g$ and $\sigma_g$ are computed by \eqref{eq:mu_sigma}, and $(\lambda_{i}, \epsilon_i(\mb{x}))$ are the first 25 eigenpairs that in total capture 95\% of spectra of the exponential correlation function
\begin{equation}\label{fig:2d_cov}
C(\mb{x}^{(1)},\mb{x}^{(2)}) = e^{-\frac{|x_1^{(1)}-x_1^{(2)}|}{L_1}-\frac{|x_2^{(1)}-x_2^{(2)}|}{L_2}}, \quad \mb{x}^{(1)},\mb{x}^{(2)}\in D.
\end{equation} 
Correlation lengths in the $x_1$ and $x_2$ directions are $L_1 = 240$ and $L_2 = 100$, which are of the order of the domain size and result in a relatively smoothly varying $\kappa(\mb{x},\bm{\xi})$ field shown in Figure \ref{fig:inverse_2d_lgCov_randomk}a.  

We choose one realization $\hat{\kappa}(\mb{x}) = \kappa(\mb{x},\bm{\xi}^*)$ as the reference diffusion coefficient and $\hat{u}(\mb{x},\bm{\xi}^*)$, the corresponding solution of Eqs \eqref{eq:diffusion_2d} and \eqref{eq:diff_2d_bd}, as the reference hydraulic head. Finite volume method with uniform mesh ($80\times 20$ elements) is adopted to solve Eqs \eqref{eq:diffusion_2d}  and \eqref{eq:diff_2d_bd}.
We consider two sets of $\hat{\kappa}$ observations, inlcuding a randomly sampled $\hat{\kappa}$ and $\hat{\kappa}$ sampled based on the expert knowledge of local maxima and minima of $\hat{\kappa}$. For each distribution of $\hat{\kappa}$ measurements, we consider two sets of $\hat{u}$ measurements, randomly located and chosen based on $\sigma^2_u(\mb{x})$ as described in Section \ref{u_sampling}. The relative error $\varepsilon$ in the estimated $\hat{\kappa}$ are computed for each case using Eq \eqref{relative_er}.

\subsubsection{Case 1: Randomly distributed locations of $\kappa$ observations} 

Here, we assume that 20 observations of $\hat{\kappa}$ are randomly distributed in the domain, as shown in Figure \ref{fig:inverse_2d_lgCov_randomk}a. The resulting conditional KL expansion of $\kappa$ has 5 unknown parameters that we estimate using 10 measurements of  $\hat{u}$. 
To study the effect of $\hat{u}$ measurement locations on the $\hat{\kappa}$ estimate errors,  we consider  measurements of $\hat{u}$ distributed (a) randomly and (b) according to $\sigma^2_u(\mb{x})$. Figure \ref{fig:inverse_2d_lgCov_randomk}b depicts the conditional $\sigma^2_u$t and the $u$ measurements locations. 
 Figure \ref{fig:inverse_2d_lgCov_case1} presents the corresponding relative errors of the estimated ${\hat{\kappa}}$.  

In both cases, the relative errors are smaller than 1\%. That is no surprise given that the number of measurements of $\kappa$ (20) is close to the number of terms in the unconditional KL representation of $\kappa$. Choosing the location of  $\hat{u}$ measurements  based on the variance of $u$  reduces $\varepsilon(\mb{x})$ by approximately the factor of 3 and the infinity norm, $\| \varepsilon (\mb{x})\|_{L^{\infty}}$, by 80\% compared to the randomly chosen locations of  $\hat{u}$ measurements.

\begin{figure}[htbp]
\centerline{
\psfig{file = 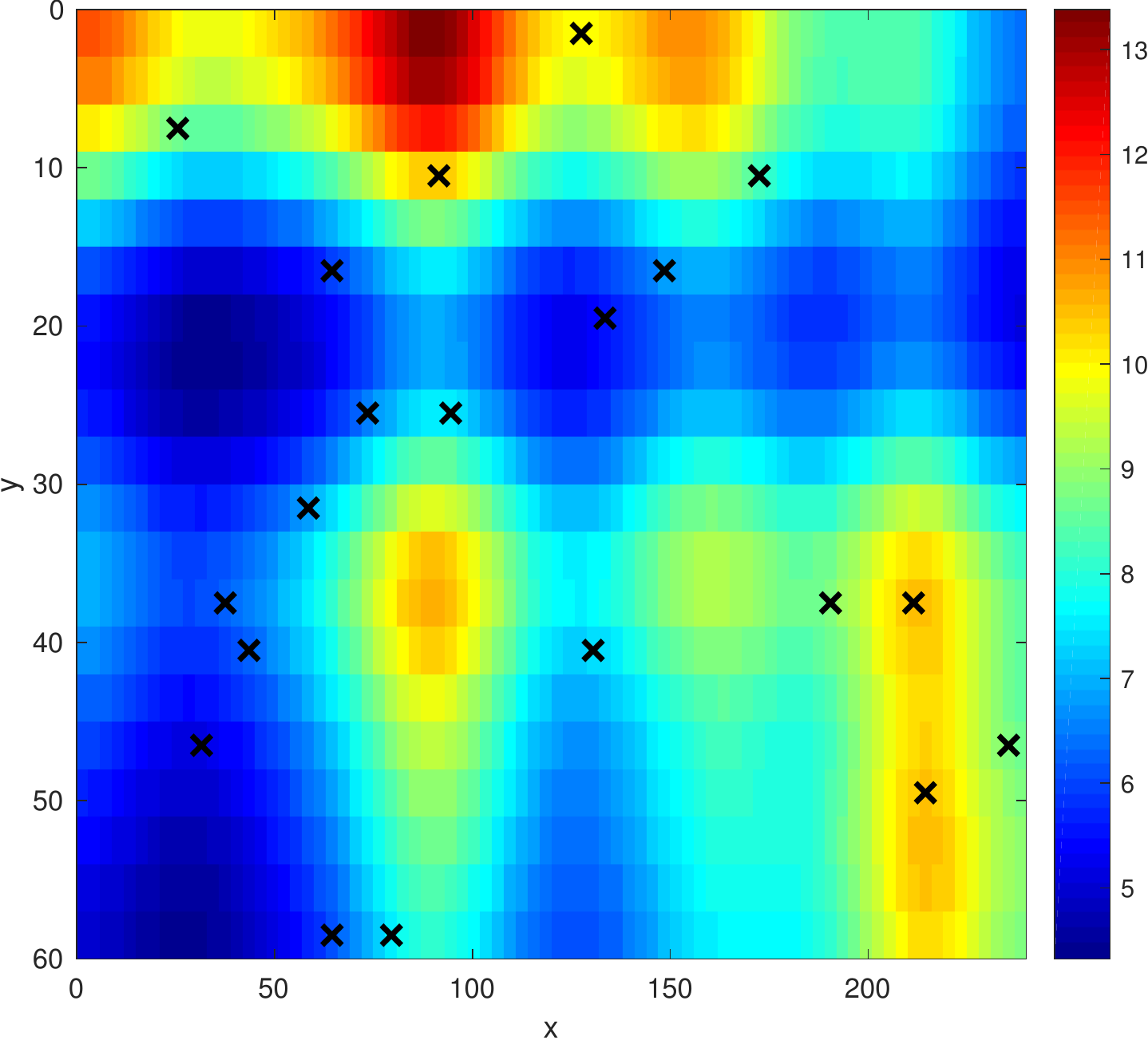, width = 7cm}
\psfig{file = 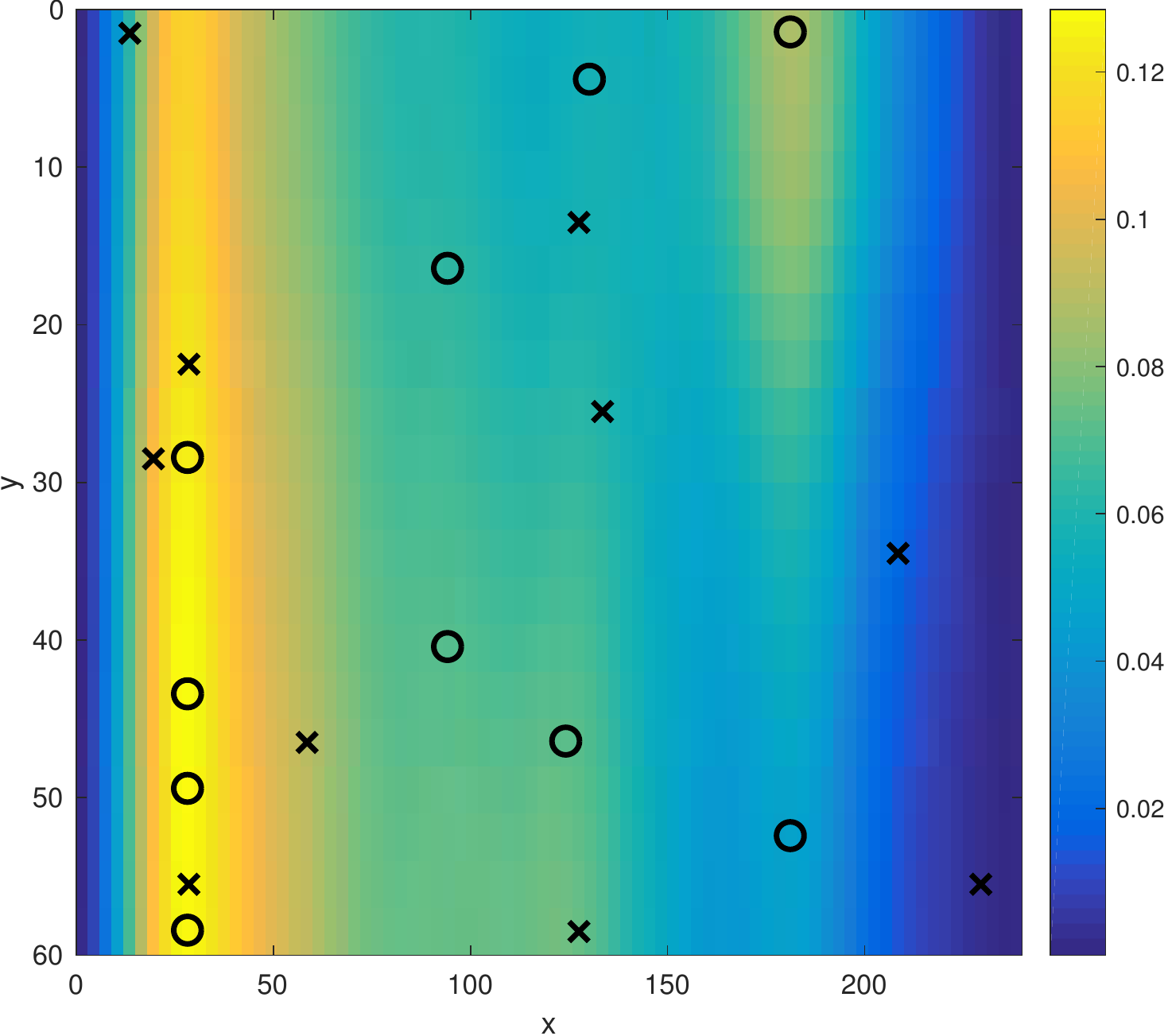, width = 7cm}
}
\caption{Two-dimensional diffusion equation, Case 1 (randomly located measurements of $\hat{\kappa}$).  (a) The color map describes the reference conductivity field $\hat{\kappa}$, and  the cross markers indicate the randomly selected locations of $\hat{\kappa}$ measurements (left). (b) The color scale denotes conditioned $\sigma^2_u$. Cross symbols denote the randomly chosen measurement locations of $\hat{u}$. Circles denote the $u$ measurement locations based on  $\sigma^2_u(\mb{x})$ (right).}
\label{fig:inverse_2d_lgCov_randomk}
\end{figure}

\begin{figure}[htbp]
\centerline{
\psfig{file = 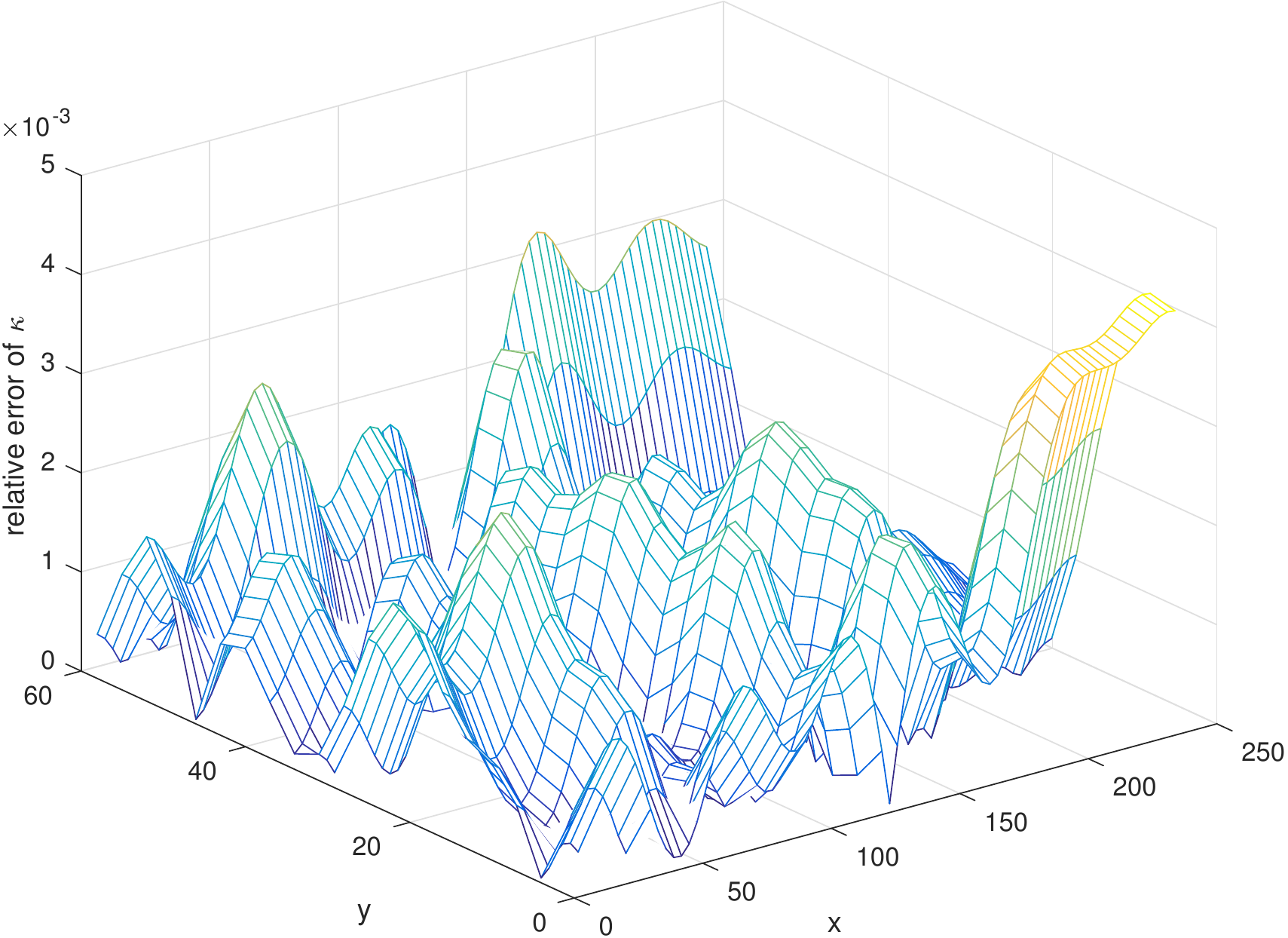, width = 7cm}
\psfig{file = 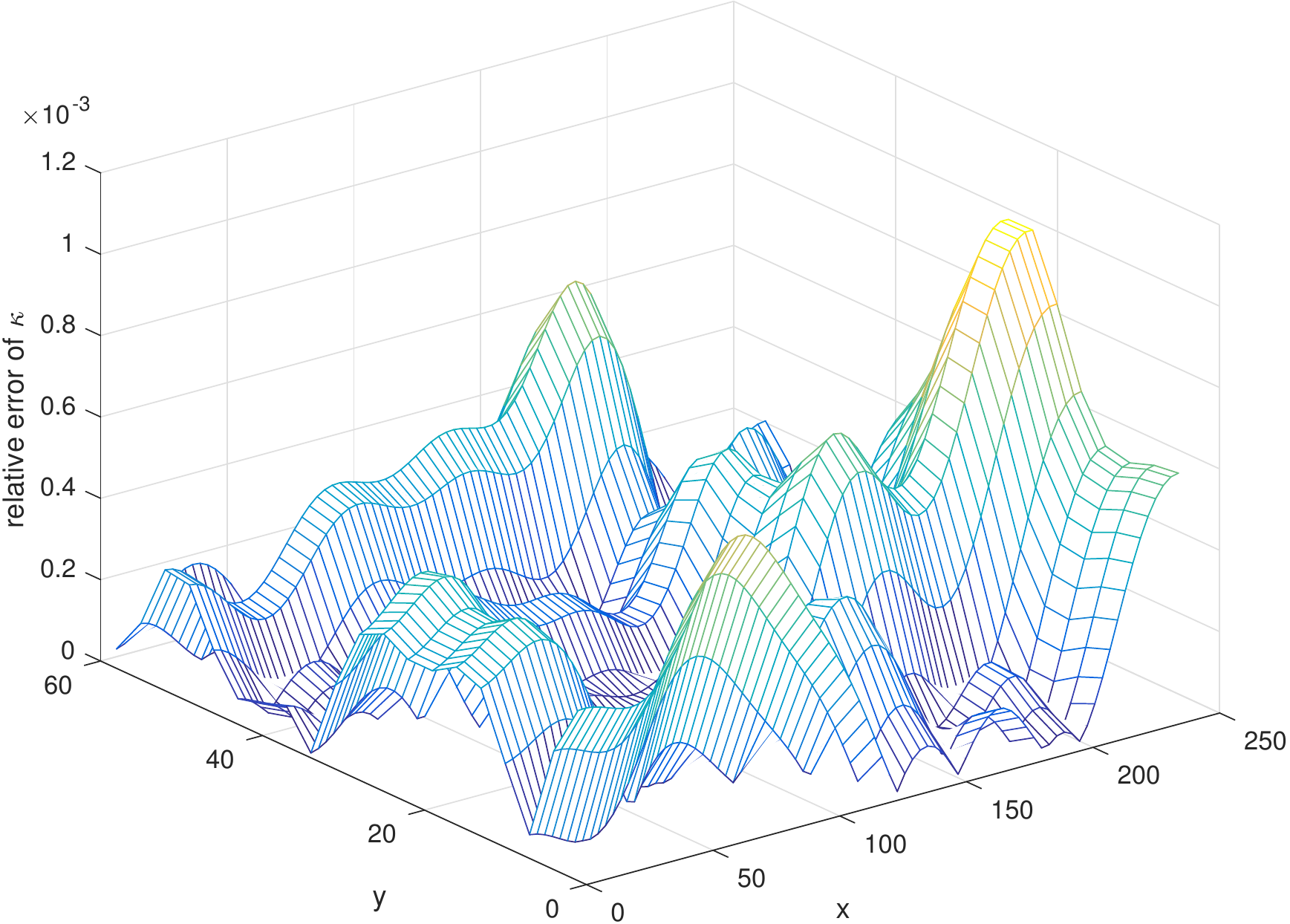, width = 7cm}
}
\caption{Two-dimensional diffusion equation, Case 1. Relative errors for (left) randomly located measurements of $\hat{u}$ and (right) measurements of $\hat{u}$ chosen based on $\sigma^2_u$.}
\label{fig:inverse_2d_lgCov_case1}
\end{figure}

\subsubsection{Case 2: Some observations of $\kappa$ are colocated with local maxima and minima of $\kappa$ field} 

Here, we assume that 9 $\hat{\kappa}$ observations are collocated with local maxima and minima of $\hat{\kappa}$,  and the other 11 $\hat{\kappa}$ observations   are randomly spread across the domain, as shown in Figure \ref{fig:inverse_2d_lgCov_selectk}a.  The resulting conditional KL expansion of $\kappa$ has 5 unknown parameters that we estimate using 10 measurements of  $\hat{u}$.  As in the previous example,  we consider  measurements of $\hat{u}$ distributed randomly and according to $\sigma^2_u(\mb{x})$, as shown in Figure \ref{fig:inverse_2d_lgCov_selectk}b.  
The corresponding relative errors in the estimated $\hat{\kappa}$ field are shown in Figure \ref{fig:inverse_2d_lgCov_case2}.

We see that for both choices of the $\hat{u}$ measurement locations,  $\varepsilon(x)$ is smaller than 1\%.  Choosing the location of $\hat{u}$ measurements based $\sigma^2_u(\mb{x})$ reduces  $\| \varepsilon \|_{L_{\infty}}$ by $25\%$ as compared with the random distribution of $\hat{u}$ measurements. 
The comparison of Figures \ref{fig:inverse_2d_lgCov_case1} and \ref{fig:inverse_2d_lgCov_case2} show that the proposed parameter estimation method performs well for both cases.    
 
\begin{figure}[htbp]
\centerline{
\psfig{file = 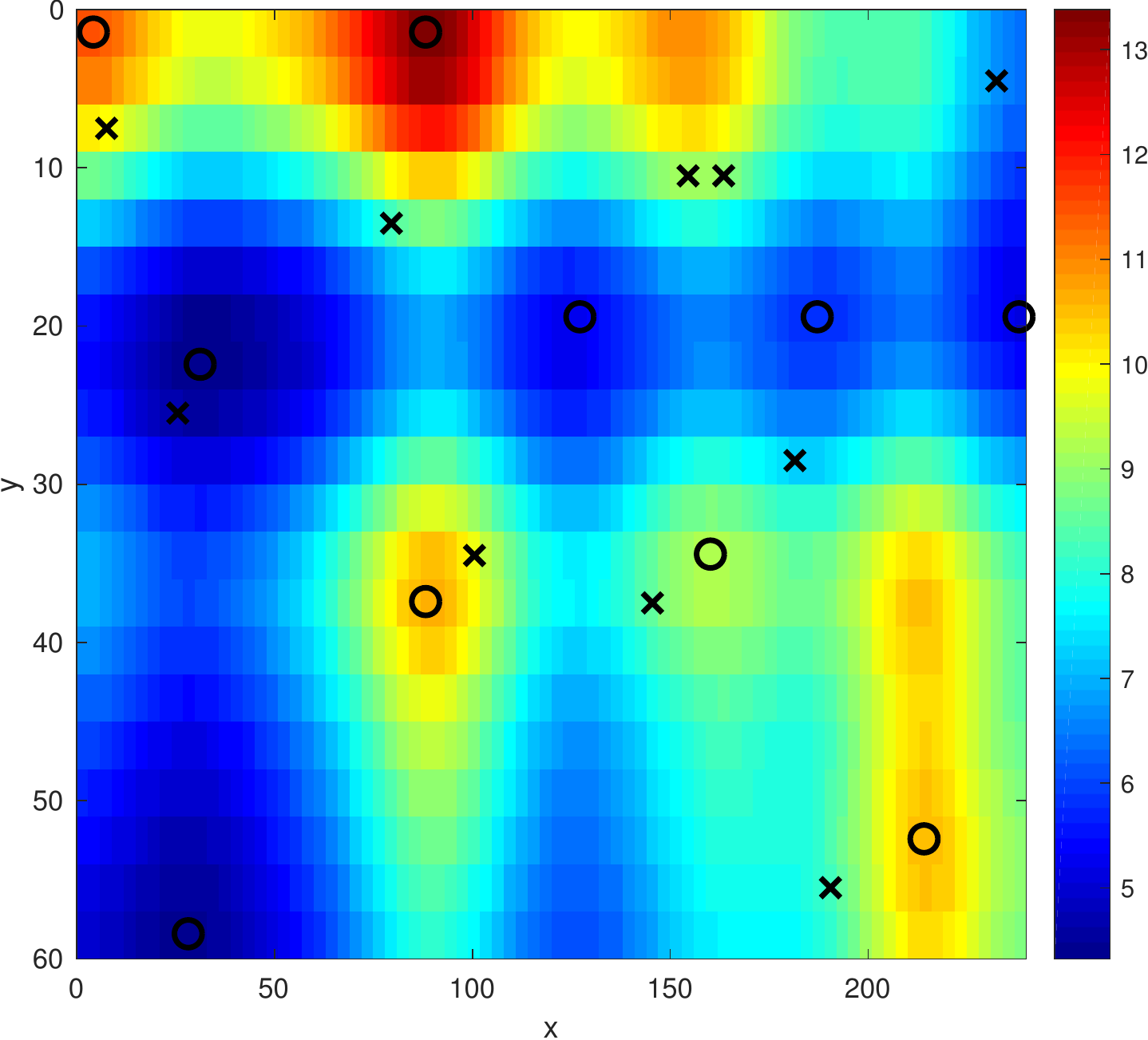, width = 7cm}
\psfig{file = 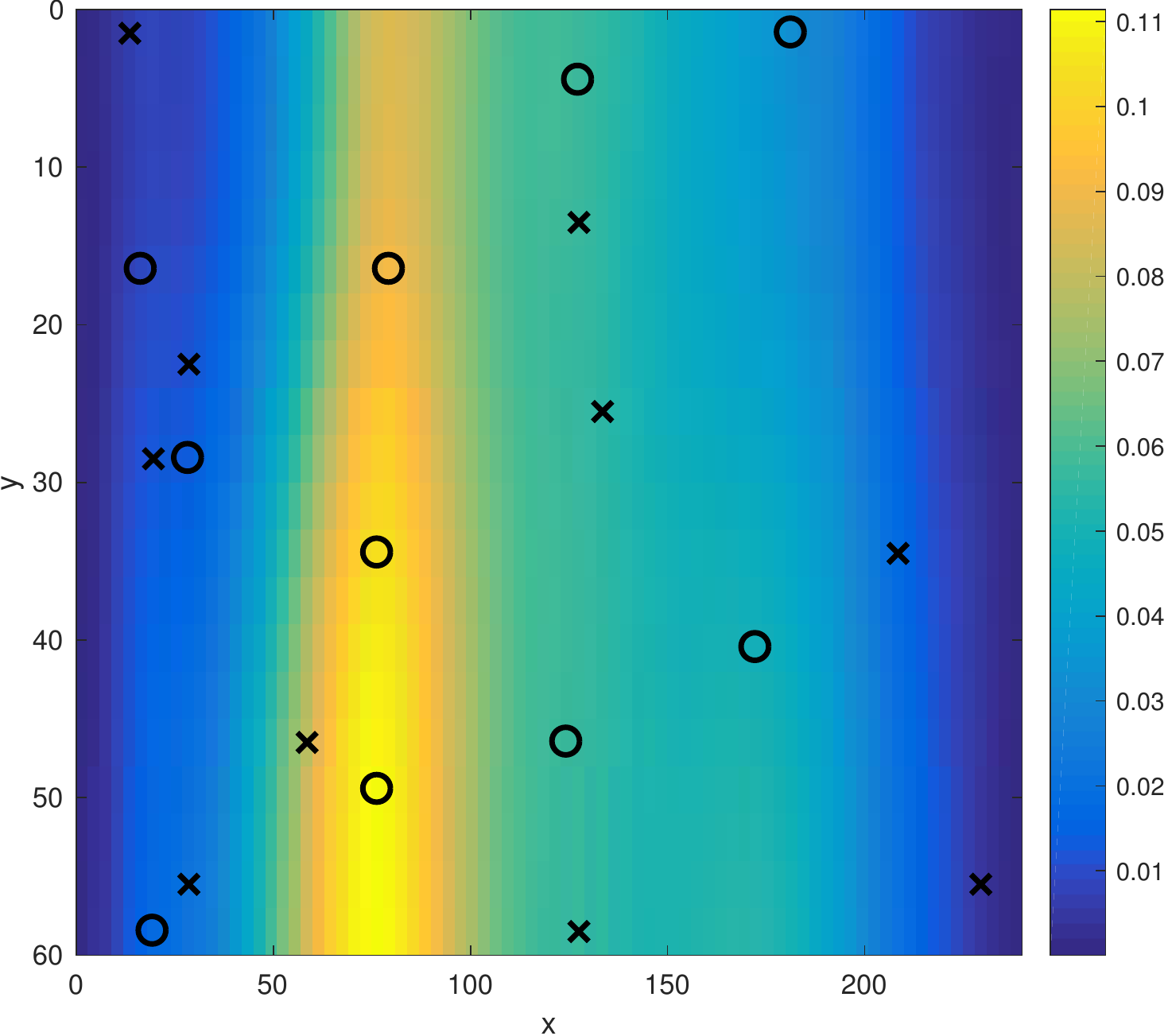, width = 7cm}
}
\caption{Example 2 Case 2 (9 selected observations of $\hat{\kappa}$ denoted by "$\circ$" markers with 11 randomly chosen observations of $\hat{\kappa}$ denoted by "$\times$" markers (left)). Two different strategies of measurement locations of $\hat{u}$: (a) cross symbols denote the randomly chosen measurement locations and (b) circle symbols denote the $u$ measurement locations based on $\sigma^2_u$ (right).}
\label{fig:inverse_2d_lgCov_selectk}
\end{figure}

\begin{figure}[htbp]
\centerline{
\psfig{file = 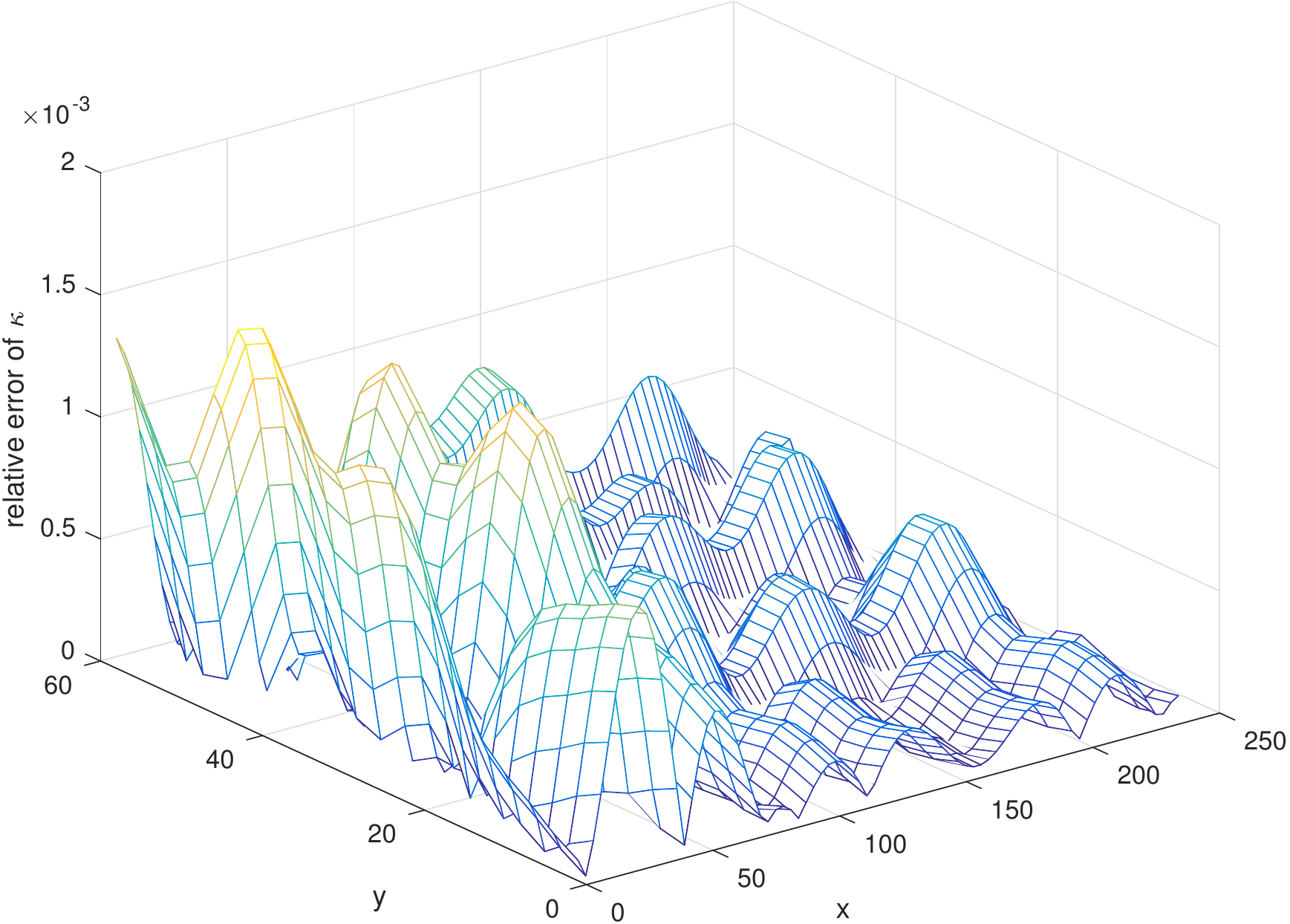, width = 7cm}
\psfig{file = 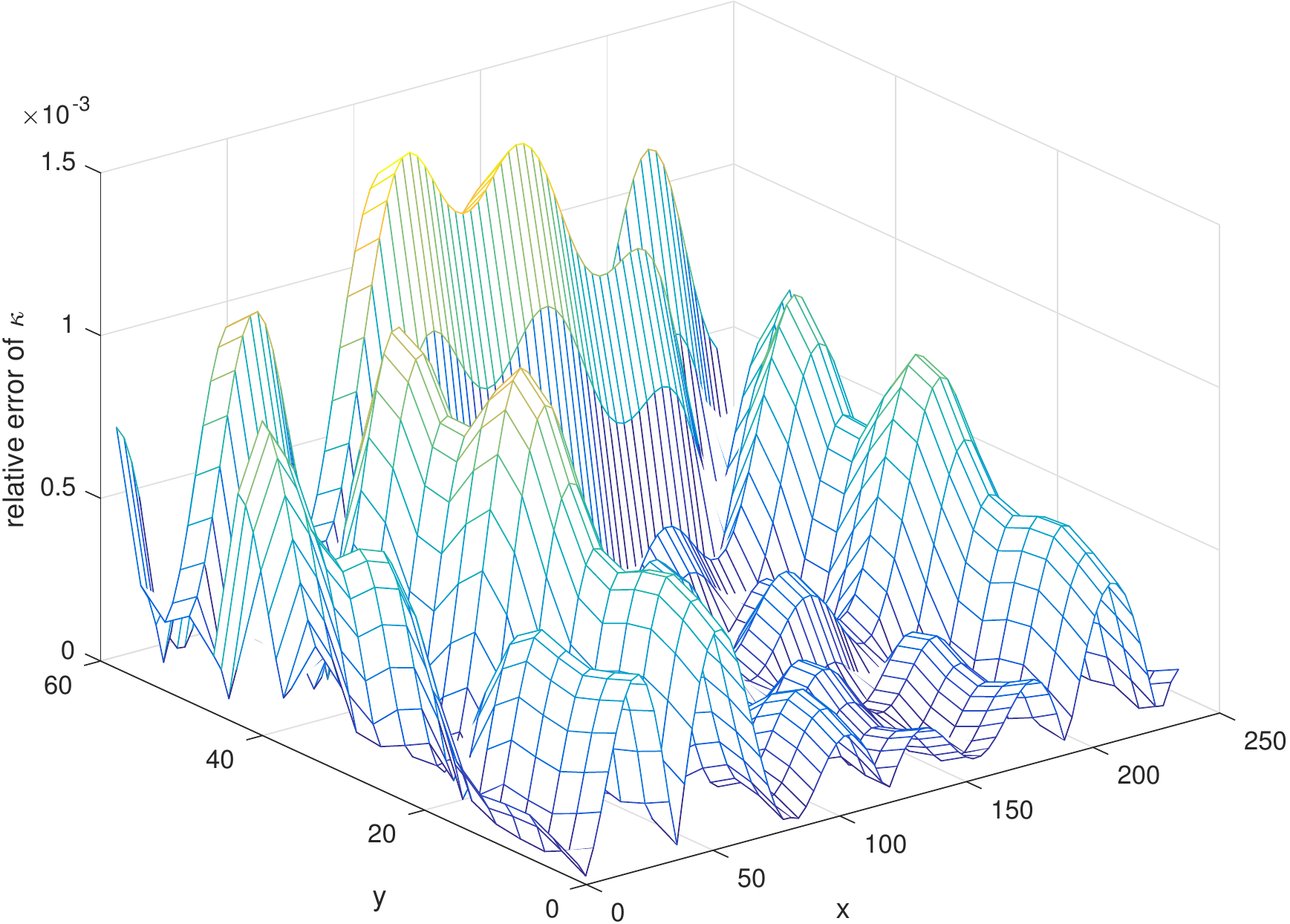 , width = 7cm}
}
\caption{Example 2 Case 2. Relative errors of the two strategies of $u$ measurement locations (left) randomly located measurements of $\hat{u}$ and (right) measurements of $\hat{u}$ taken based on $\sigma^2_u$.}
\label{fig:inverse_2d_lgCov_case2}
\end{figure}

\subsection{Two-dimensional diffusion equation with small correlation length of $\kappa$ (rough $\kappa$ field)}

In this example, we assume the exponential covariance function of $ g(x,\omega) = \log \kappa(x,\omega)$:
\begin{equation}
C(\mb{x}^{(1)},\mb{x}^{(2)}) = e^{-\frac{|\mb{x}^{(1)}-\mb{x}^{(2)}|^2}{L^2}}, \quad \mb{x}^{(1)},\mb{x}^{(2)}\in D, 
\end{equation} 
where $|\mb{x}^{(1)}-\mb{x}^{(2)}|$ is the Euclidean distance between $\mb{x}^{(1)}$ and $\mb{x}^{(2)}$ and $L= 0.1$. The mean and variance of $\kappa(\mb{x},\omega)$ are set to $\mu_k = 5$ and $\sigma_k = 2.5$.

We solve the two-dimensional equation \eqref{eq:diffusion_2d} on $D = [0,2]\times [0,1] $ subject to the boundary conditions:
\begin{equation}\label{eq:diff_2d_bd2}
\begin{split}
&u(0,x_2,\omega) = 2,\quad u(2,x_2,\omega) = 0,\\
&-\mb{n}\cdot (\kappa(\mb{x},\omega)\nabla u)|_{(x_1,0)} = \mb{n}\cdot (\kappa(\mb{x},\omega)\nabla u)|_{(x_1,1)} = 0, 
\end{split}
\end{equation}
where $\mb{n} = (0,1)$.

Here, the ratio of the correlation length to the domain size is on the order of 0.1, which is smaller than in the previous case and results in a rougher $\kappa$. The KL representation of the unconditional random field $g(\mb{x},\omega)$ requires  $210$ terms to capture 95\% of spectrum,

\begin{equation}\label{KL210}
g(\mb{x},\bm{\xi}) = \mu_g+ \sigma_g\sum_{i=1}^{210}\sqrt{\lambda_{i}}\epsilon_i(x)\xi_i.
\end{equation}

The reference conductivity field $\kappa(\bm{x},\mb{\xi}^*)$ is constructed as a realization of $\kappa(\bm{x},\mb{\xi})$ and is shown in Figure \ref{fig:inverse_2d_stCov_randomk}a, which is significantly rougher than the field in Figure \ref{fig:inverse_2d_lgCov_randomk}a considered in the previous case.

The corresponding reference field $\hat{u}(\bm{x},\mb{\xi}^*)$ is found as the solution of Eq \eqref{eq:diffusion_2d} subject to the boundary conditions \eqref{eq:diff_2d_bd2}. This equation is solved using the finite volume method with equal-distanced mesh with $128\times 64$ elements. 
We assume that 205 measurements of $\hat{\kappa}$ and 10 measurements of $\hat{u}$ are available to estimate the entire field $\hat{\kappa}(x)$. The KL expansion (\ref{KL210}), conditioned on 205 measurements, has five random dimensions. As before, we study the effect of $\hat{\kappa}$ and $\hat{u}$ measurements locations distribution on the accuracy of $\kappa$ estimation. 

\subsubsection{Case 1: Randomly located $\kappa$ observations} 

Here, we assume that the locations of $\hat{\kappa}$ measurements are randomly distributed, as shown in Figure \ref{fig:inverse_2d_stCov_randomk}a. Next we consider 10 $\hat{u}$ measurement locations distributed in two ways, including randomly distributed measurements and measurements distributed according to the conditional $\sigma^2_u(\mb{x})$; both distributions are depicted in Figure \ref{fig:inverse_2d_stCov_randomk}b.  Figure \ref{fig:inverse_2d_stCov_case1} presents the corresponding relative error of inferred 
$\hat{\kappa}(x)$.

The comparison of Figures \ref{fig:inverse_2d_stCov_case1}a and b show that choosing $\hat{u}$ measurement locations according to the $u$ variance reduces the $\| \varepsilon \|_{L^{\infty}}$ error by a factor of 10. 

\begin{figure}[htbp]
\centerline{
\psfig{file = 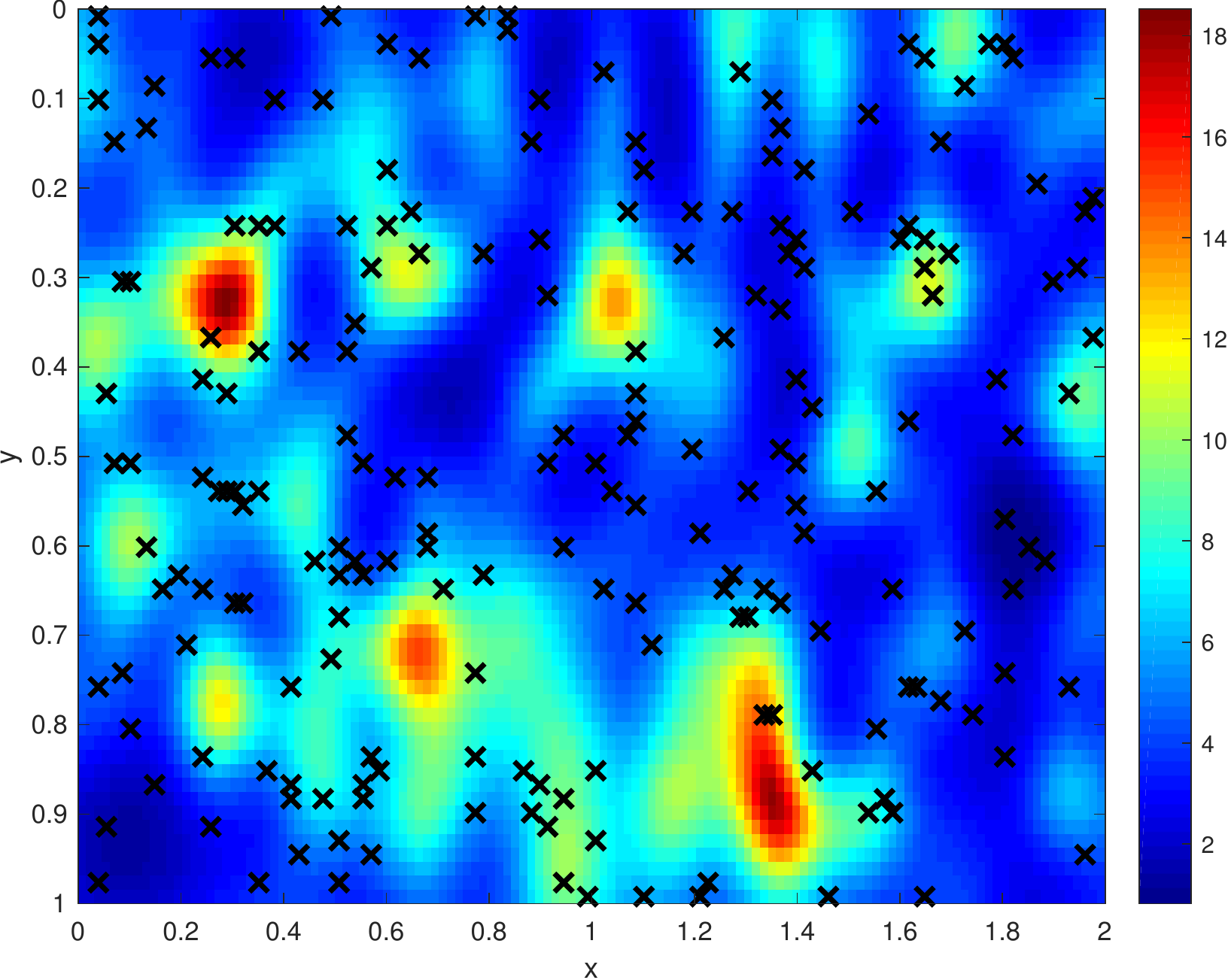, width = 7cm}
\psfig{file = 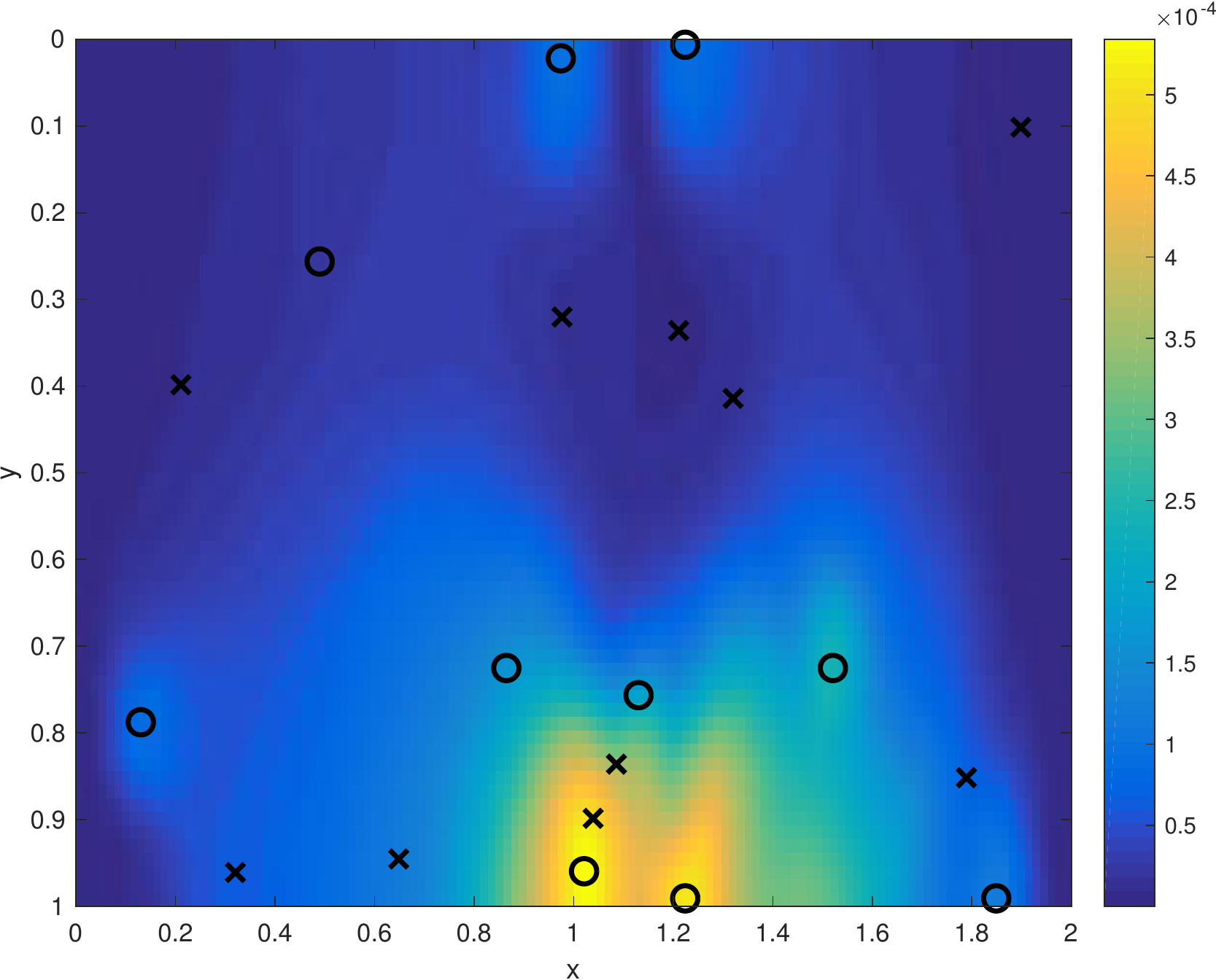, width = 7cm}
}
\caption{Example 3 Case 1(random observations of $\hat{\kappa}$ (left)). Two strategies of selecting $u$ measurement locations: (a) randomly collocated on the variance surface of $u$ denoted by cross symbols (b) collocated based on $\sigma_u^2$ denoted by circles (right).}
\label{fig:inverse_2d_stCov_randomk}
\end{figure}

\begin{figure}[htbp]
\centerline{
\psfig{file = 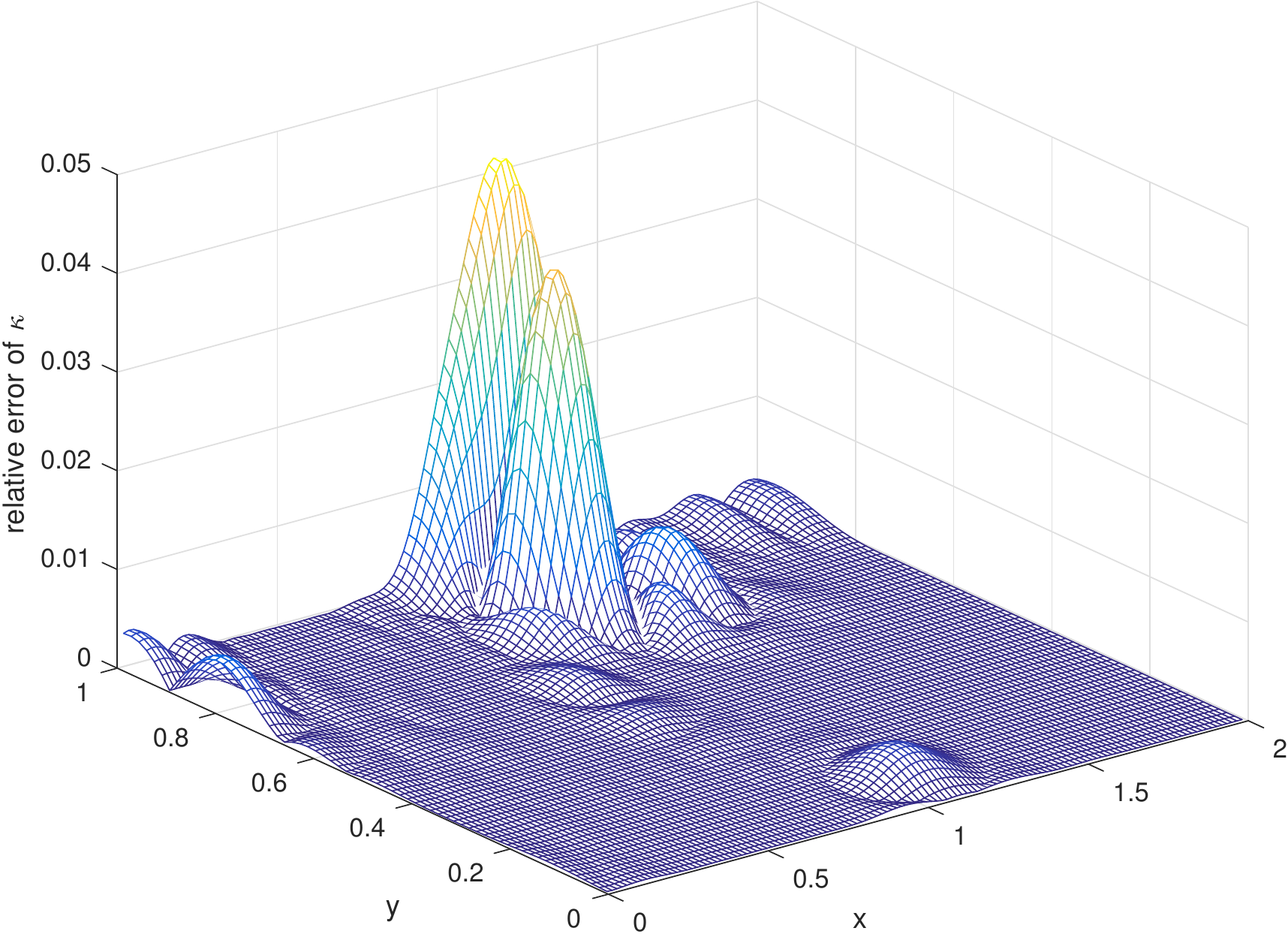, width = 7cm}
\psfig{file = 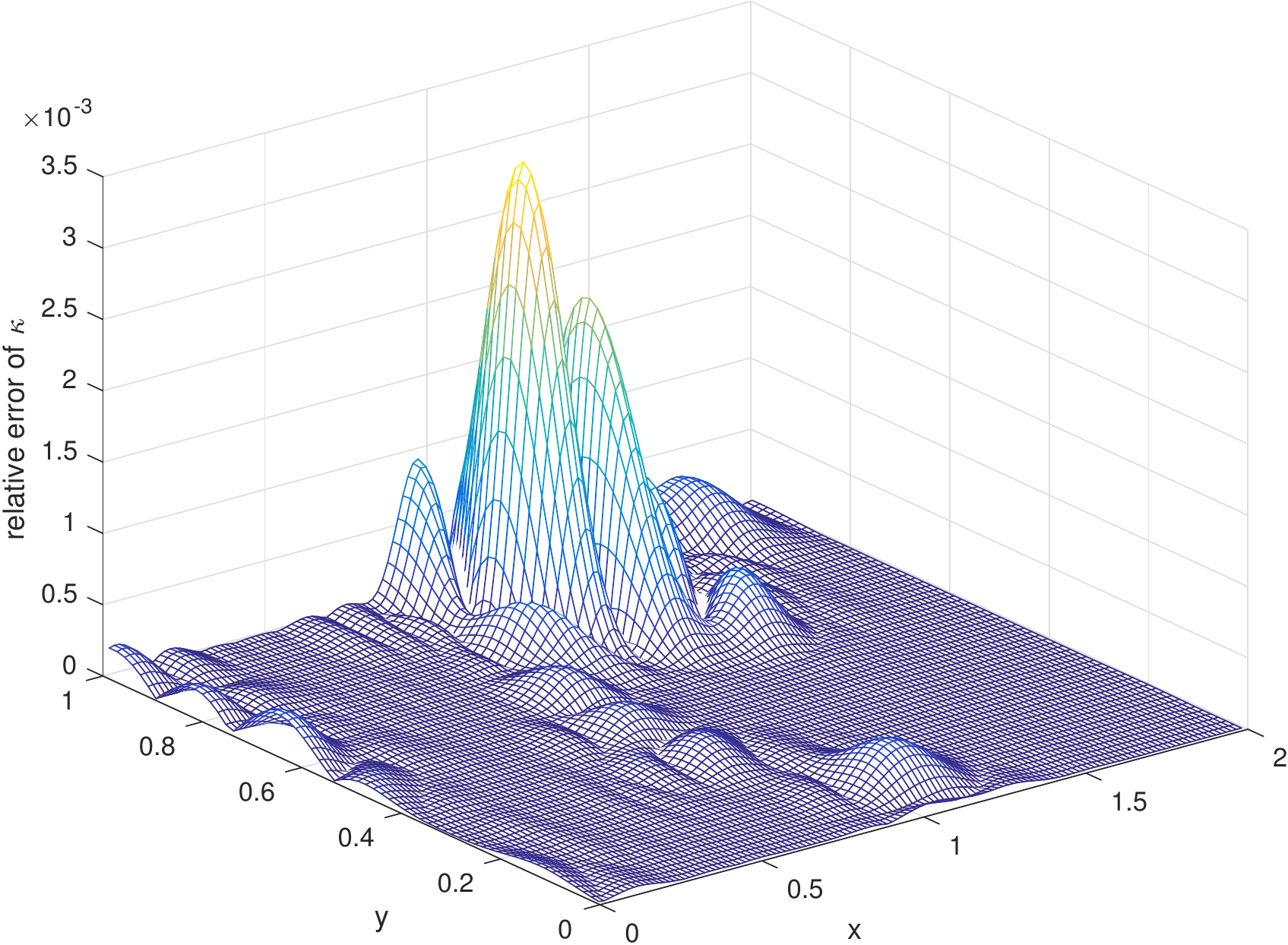 , width = 7cm}
}
\caption{Example 3 Case 1. Relative errors of the two strategies of $u$ measurement locations (left) randomly located measurements of $\hat{u}$ and (right) measurements of $\hat{u}$ taken based on $\sigma_u^2$.}
\label{fig:inverse_2d_stCov_case1}
\end{figure}

\subsubsection{Case 2: Some observations of $\kappa$ are colocated with local maxima and minima of $\kappa$ field} 
Here, we assume that 50 locations of $\hat{\kappa}$ measurements  are collocated with local maxima, minima, and saddle points of $\hat{\kappa}$ and, $155$ measurement locations are randomly distributed, as shown in Figure \ref{fig:inverse_2d_stCov_selectk}a. As before, we consider two choices of $\hat{u}$ measurement locations, including random locations and locations based on the conditional $\sigma^2_u(\mb{x})$ as shown in Figure \ref{fig:inverse_2d_stCov_selectk}b. Figure \ref{fig:inverse_2d_stCov_selectk}b also displays the conditional  $\sigma^2_u(\mb{x})$ as a function of $\mb{x}$.
Figure \ref{fig:inverse_2d_stCov_case2} presents the corresponding relative error of the inferred $\hat{\kappa}(\mb{x})$.

The comparison of Cases 1 and 2 shows that for randomly located $\hat{\kappa}$ measurements, the selection of $u$ measurement locations based on the conditional $u$ variance is very important -- it reduces the $\| \varepsilon (\mb{x})\|_{L_{\infty}}$ error by a factor of 10. For specially selected locations of $\hat{\kappa}$, the $\sigma^2_u(\mb{x})$-based selection of $\hat{u}$ measurement locations reduces $\| \varepsilon (\mb{x})\|_{L_{\infty}}$ by a factor of 0.4.    

\begin{figure}[htbp]
\centerline{
\psfig{file = 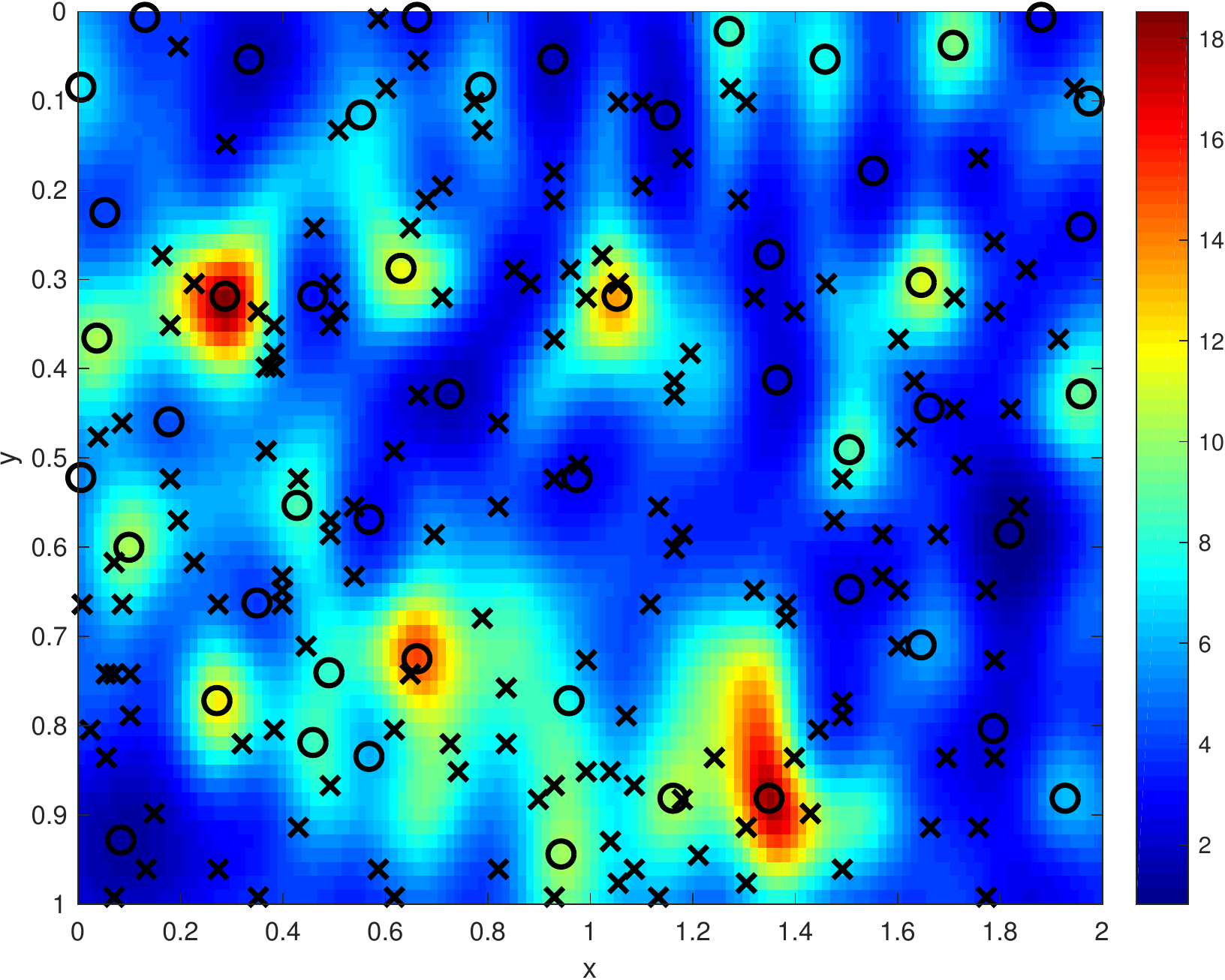, width = 7cm}
\psfig{file = 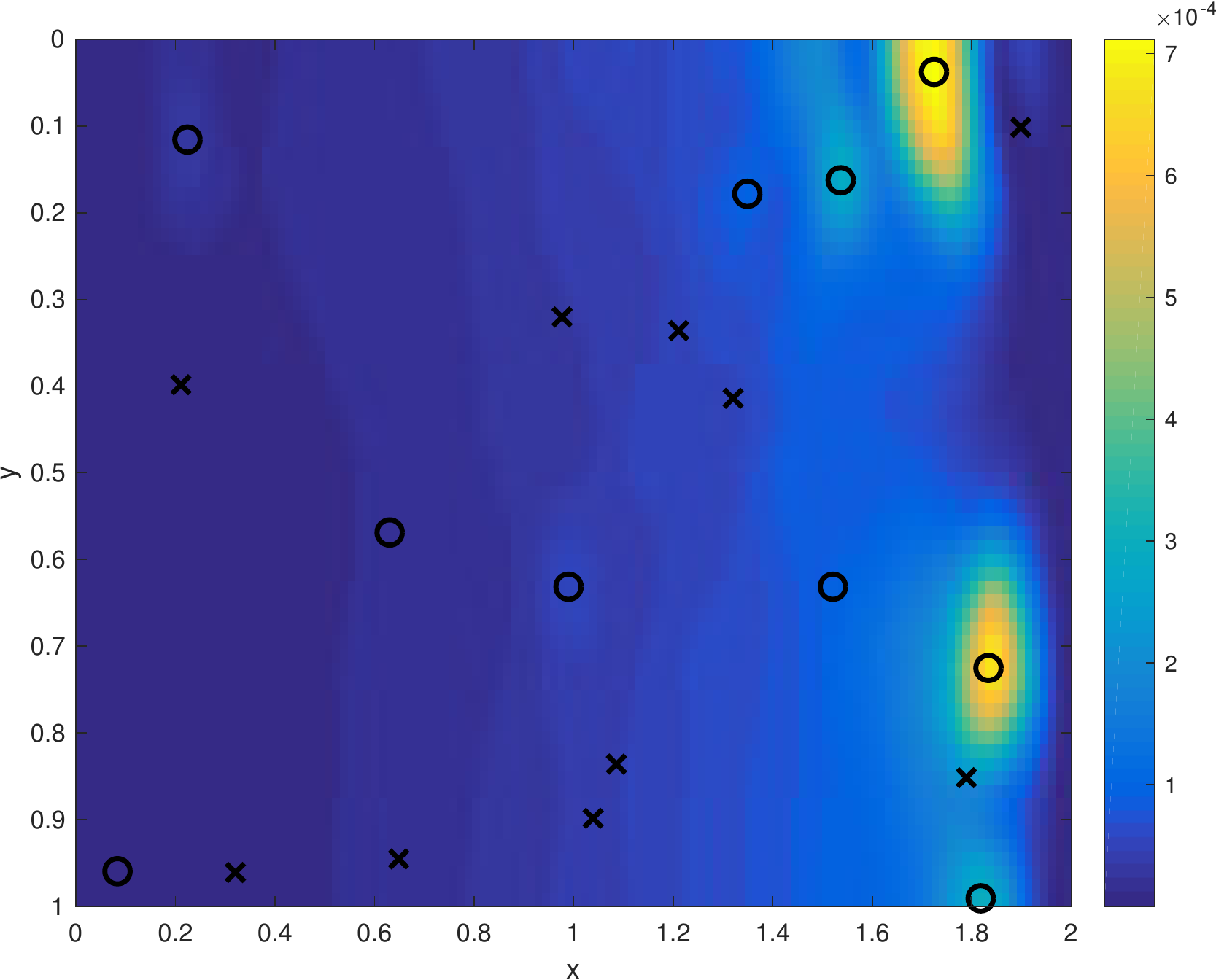, width = 7cm}
}
\caption{Example 3 Case 2 (selected observations of $\hat{\kappa}$ (left)). Two strategies of selecting $u$ measurement locations: (a) randomly collocated on the variance surface of $u$ denoted by cross symbols (b) collocated based on $\sigma_u^2$ denoted by circles (right).}
\label{fig:inverse_2d_stCov_selectk}
\end{figure}

\begin{figure}[htbp]
\centerline{
\psfig{file = 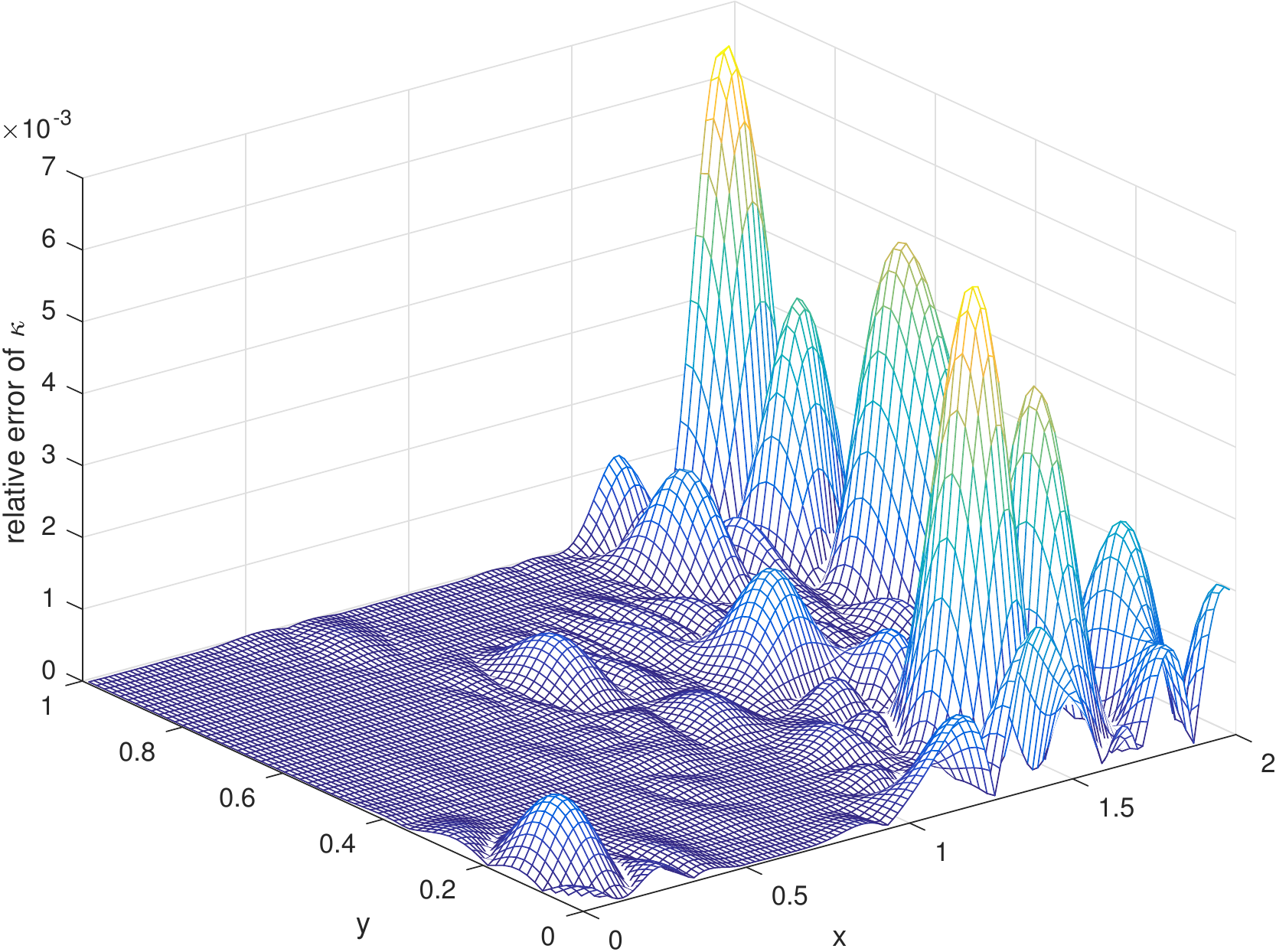, width = 7cm}
\psfig{file = 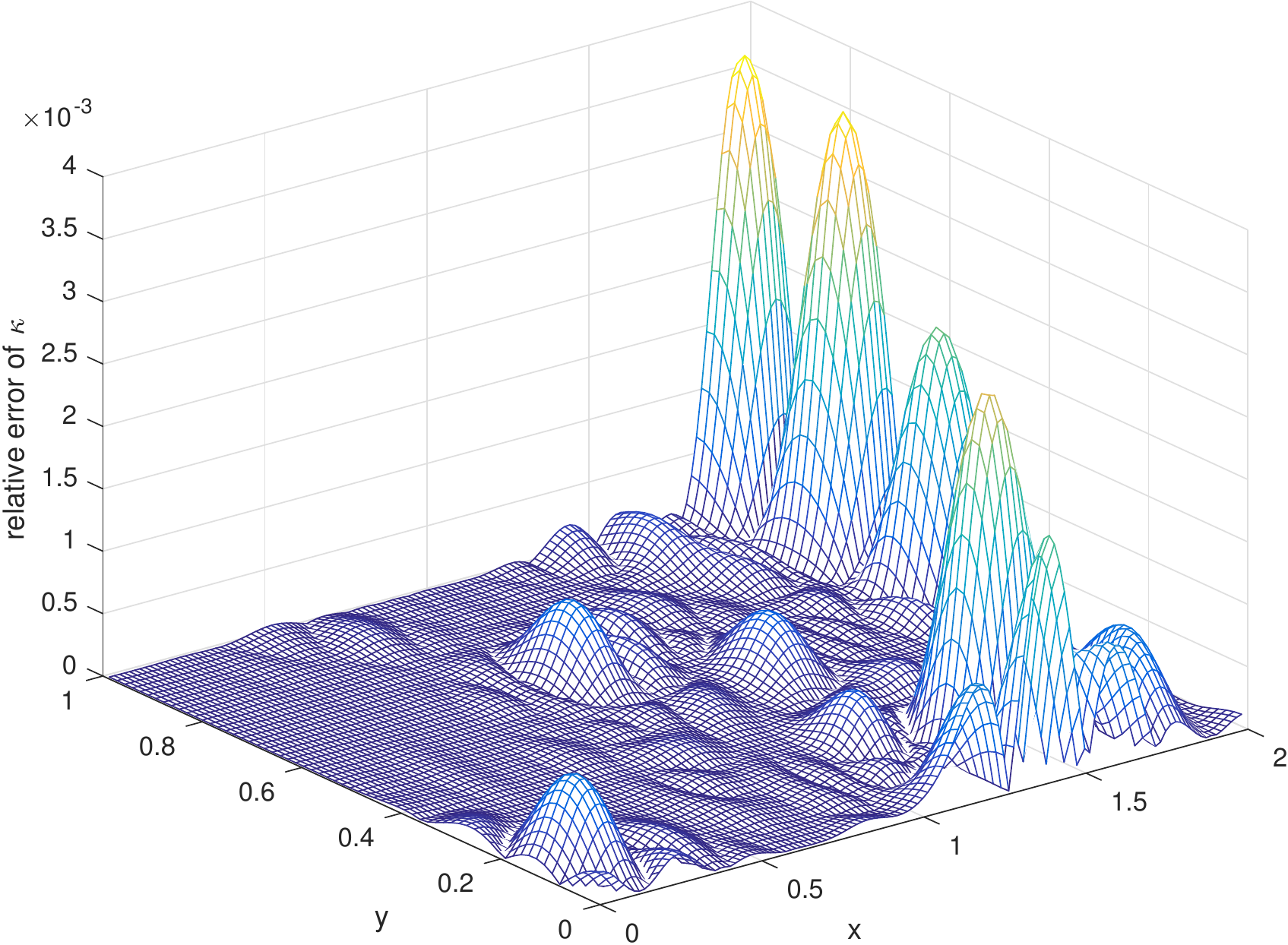 , width = 7cm}
}
\caption{Example 3 Case 2. Relative errors of the two strategies of $u$ measurement locations (left) randomly located measurements of $\hat{u}$ and (right) measurements of $\hat{u}$ taken based on $\sigma^2_u$.}
\label{fig:inverse_2d_stCov_case2}
\end{figure}

\begin{remark}\label{con_R1}
In all considered examples, we see a correlation between the error of the estimated $\kappa$ and the variance of $u$. That is, near the area that the gPC solution has large variance, the inferred $\kappa$ usually has larger relative error. 
\end{remark}

\section{Conclusion}\label{sec:conclusion}
In this work, we proposed a conditional KL expansion and gPC surrogate model for estimating space-dependent coefficients  in partial differential equation models using measurements of the coefficient and state variable.
We demonstrated that the conditional gPC model simplifies the optimization problem and reduces the number of unknown parameters as compared to parameter estimation using traditional (unconditional) gPC surrogates.  
Specifically,  the conditional gPC reduces dimensionality of the parameter space and replaces expensive direct solutions of PDEs with the conditional gPC surrogate.  Also, in the absence of measurement error, the estimated parameter field exactly matches the parameter measurements. 
Furthermore, we proposed using the surrogate model to determine measurement locations based on the variance of state variable conditioned on the coefficient measurements. Specifically, we proposed collocating measurement with the local maxima of the state variable variance predicted by the conditional gPC. 
We presented one- and two-dimensional examples demonstrating the overall accuracy of the proposed approach (it is more accurate than Gaussian process regression). We also showed that the error in the estimated parameter is smaller when the state variable measurements are chosen based on its variance rather than uniformly or randomly collected. Also, we demonstrated that selecting the coefficient measurement locations based on physical knowledge (i.e., collocating coefficient measurements with local maxima and minima of the coefficient) further reduces error in the estimated coefficient.    

Finally, in all considered examples, we see  correlation between the error of the estimated coefficient and the variance of the solution; that is, near the area that the gPC solution has large variance, the inferred coefficient has larger relative error. Therefore, the conditional variance of state variables can be used to guide the data acquisition for the coefficient.  

 In the current work, we only investigated cases where the unknown coefficient lies in a function space of the finite KL modes (i.e., the reference solution is a realization of the known finite-dimensional Gaussian process) and the number of solution measurements is greater than the dimension of the conditional KL representation. Problems with small data sets and coefficients with unknown and/or non-Gaussian distribution will yield additional errors and uncertainty in parameter estimation and will be subject of our future work.
  
\section*{Acknowledgements}
This work was supported by the U.S. Department of Energy (DOE) Office of Science, Office of Advanced Scientific Computing Research. Pacific Northwest National Laboratory is operated by Battelle for the DOE under Contract DE-AC05-76RL01830.

\bibliography{random.bib,sc-uq-DCATL1.bib}

\begin{thebibliography}{10}

\bibitem{beck1985inverse}
James~V Beck, Ben Blackwell, and Charles R~St Clair~Jr.
\newblock {\em Inverse heat conduction: Ill-posed problems}.
\newblock James Beck, 1985.

\bibitem{Chen2016}
Peng Chen and Christoph Schwab.
\newblock {\em Adaptive Sparse Grid Model Order Reduction for Fast Bayesian
  Estimation and Inversion}, pages 1--27.
\newblock Springer International Publishing, Cham, 2016.

\bibitem{Cheney1990inverse}
M.~Cheney, D.~Isaacson, J.~C. Newell, S.~Simske, and J.~Goble.
\newblock Noser: An algorithm for solving the inverse conductivity problem.
\newblock {\em International Journal of Imaging Systems and Technology},
  2(2):66--75, 1990.

\bibitem{Christen2005}
J.~Andr\'{e}s Christen and Colin Fox.
\newblock Markov chain monte carlo using an approximation.
\newblock {\em Journal of Computational and Graphical Statistics},
  14(4):795--810, 2005.

\bibitem{Cressie1993}
Noel Cressie.
\newblock {\em Statistics for Spatial Data, Revised Edition}.
\newblock Wiley-Interscience, 1993.

\bibitem{Cui2015}
Tiangang Cui, Youssef~M. Marzouk, and Karen~E. Willcox.
\newblock Data-driven model reduction for the bayesian solution of inverse
  problems.
\newblock {\em International Journal for Numerical Methods in Engineering},
  102(5):966--990, 2015.

\bibitem{de1986quantitative}
Ghislain De~Marsily.
\newblock Quantitative hydrogeology.
\newblock Technical report, Paris School of Mines, Fontainebleau, 1986.

\bibitem{Frangos2010}
M.~Frangos, Y.~Marzouk, K.~Willcox, and B.~van Bloemen~Waanders.
\newblock {\em Surrogate and Reduced-Order Modeling: A Comparison of Approaches
  for Large-Scale Statistical Inverse Problems}, pages 123--149.
\newblock John Wiley \& Sons, Ltd, 2010.

\bibitem{gamerman2006MCMC}
Dani Gamerman and Hedibert~F Lopes.
\newblock {\em Markov chain Monte Carlo: stochastic simulation for Bayesian
  inference}.
\newblock Chapman and Hall/CRC, 2006.

\bibitem{GhanemS91}
R.G. Ghanem and P.~Spanos.
\newblock {\em Stochastic Finite Elements: a Spectral Approach}.
\newblock Springer-Verlag, 1991.

\bibitem{Hampton_2015Cs}
J.~Hampton and A.~Doostan.
\newblock Compressive sampling of polynomial chaos expansions: Convergence
  analysis and sampling strategies.
\newblock {\em J. Comput. Phys}, 280:363--386, 2015.

\bibitem{Hanke_1997}
Martin Hanke.
\newblock A regularizing levenberg - marquardt scheme, with applications to
  inverse groundwater filtration problems.
\newblock {\em Inverse Problems}, 13(1):79--95, feb 1997.

\bibitem{jakemanG2013}
John~D Jakeman and Stephen~G Roberts.
\newblock Local and dimension adaptive stochastic collocation for uncertainty
  quantification.
\newblock In {\em Sparse grids and applications}, pages 181--203. Springer,
  2013.

\bibitem{lih_2014}
Heng Li.
\newblock Conditional simulation of flow in heterogeneous porous media with the
  probabilistic collocation method.
\newblock {\em Communications in Computational Physics}, 16(4):1010¨C1030,
  2014.

\bibitem{ZhangDX_KL}
Gaisheng Liu, Zhiming Lu, and Dongxiao Zhang.
\newblock Stochastic uncertainty analysis for solute transport in randomly
  heterogeneous media using a karhunen-lo\'{e}ve-based moment equation
  approach.
\newblock {\em Water Resources Research}, 43(7), 2007.

\bibitem{Loeve1978}
M.~Lo\'{e}ve.
\newblock {\em Probability Theory, vols. II}.
\newblock Springer, New York, 1978.

\bibitem{Ma2009}
X.~Ma and N.~Zabaras.
\newblock An adaptive hierarchical sparse grid collocation algorithm for the
  solution of stochastic differential equations.
\newblock {\em J. Comput. Phys.}, 228(8):3084--3113, 2009.

\bibitem{MarzoukXiu2009}
Youssef Marzouk and Dongbin Xiu.
\newblock A stochastic collocation approach to bayesian inference in inverse
  problems.
\newblock {\em Communications in Computational Physics}, 6(4):826--847, 10
  2009.

\bibitem{MARZOUK2009}
Youssef~M. Marzouk and Habib~N. Najm.
\newblock Dimensionality reduction and polynomial chaos acceleration of
  bayesian inference in inverse problems.
\newblock {\em Journal of Computational Physics}, 228(6):1862 -- 1902, 2009.

\bibitem{MARZOUK2007}
Youssef~M. Marzouk, Habib~N. Najm, and Larry~A. Rahn.
\newblock Stochastic spectral methods for efficient bayesian solution of
  inverse problems.
\newblock {\em Journal of Computational Physics}, 224(2):560 -- 586, 2007.

\bibitem{mclaughlin1996reassessment}
Dennis McLaughlin and Lloyd~R Townley.
\newblock A reassessment of the groundwater inverse problem.
\newblock {\em Water Resources Research}, 32(5):1131--1161, 1996.

\bibitem{mercer1909xvi}
James Mercer.
\newblock Xvi. functions of positive and negative type, and their connection
  the theory of integral equations.
\newblock {\em Philosophical transactions of the royal society of London.
  Series A, containing papers of a mathematical or physical character},
  209(441-458):415--446, 1909.

\bibitem{Nguyen2014}
Ngoc-Hien Nguyen, Boo Cheong~Khoo, and Karen Willcox.
\newblock Model order reduction for bayesian approach to inverse problems.
\newblock {\em Asia Pacific Journal on Computational Engineering}, 1(1):2, Apr
  2014.

\bibitem{nobile2008sparse}
Fabio Nobile, Ra{\'u}l Tempone, and Clayton~G Webster.
\newblock A sparse grid stochastic collocation method for partial differential
  equations with random input data.
\newblock {\em SIAM J. Numer. Anal.}, 46(5):2309--2345, 2008.

\bibitem{ossiander2014}
Mina~E. Ossiander, Malgorzata Peszynska, and Veronika~S. Vasylkivska.
\newblock Conditional stochastic simulations of flow and transport with
  karhunen-loe've expansions, stochastic collocation, and sequential gaussian
  simulation.
\newblock {\em J. Appl. Math.}, 2014:21 pages, 2014.

\bibitem{Schwab2006}
Christoph Schwab and Radu~Alexandru Todor.
\newblock Karhunen-lo\`{e}ve approximation of random fields by generalized fast
  multipole methods.
\newblock {\em J. Comput. Phys.}, 217(1):100--122, September 2006.

\bibitem{stuart15c}
A.~M. Stuart.
\newblock Inverse problems: A bayesian perspective.
\newblock {\em Acta Numerica}, 19:451¨C559, 2010.

\bibitem{Tong1990}
Y.L. Tong.
\newblock {\em The Multivariate Normal Distribution}.
\newblock 0172-7397. Springer-Verlag New York, 1 edition, 1990.

\bibitem{VRUGT2016}
Jasper~A. Vrugt.
\newblock Markov chain monte carlo simulation using the dream software package:
  Theory, concepts, and matlab implementation.
\newblock {\em Environmental Modeling \& Software}, 75:273 -- 316, 2016.

\bibitem{Xiu_CICP07}
D.~Xiu.
\newblock Efficient collocational approach for parametric uncertainty analysis.
\newblock {\em Comm. Comput. Phys.}, 2(2):293--309, 2007.

\bibitem{Xiu_CICP09}
D.~Xiu.
\newblock Fast numerical methods for stochastic computations: a review.
\newblock {\em Comm. Comput. Phys.}, 5:242--272, 2009.

\bibitem{Xiu10}
D.~Xiu.
\newblock {\em Numerical methods for stochastic computations}.
\newblock Princeton Univeristy Press, Princeton, New Jersey, 2010.

\bibitem{XiuH_SISC05}
D.~Xiu and J.S. Hesthaven.
\newblock High-order collocation methods for differential equations with random
  inputs.
\newblock {\em SIAM J. Sci. Comput.}, 27(3):1118--1139, 2005.

\bibitem{XiuK_SISC02}
D.~Xiu and G.E. Karniadakis.
\newblock The {Wiener-Askey} polynomial chaos for stochastic differential
  equations.
\newblock {\em SIAM J. Sci. Comput.}, 24(2):619--644, 2002.

\bibitem{Yan_2012Sc}
L.~Yan, L.~Guo, and D.~Xiu.
\newblock Stochastic collocation algorithms using $l^1$ minimization.
\newblock {\em Inter. J. Uncertain Quantification}, 2:279--293, 2012.

\bibitem{Yang_2013reweightedL1}
X.~Yang and G.~E. Karniadakis.
\newblock Reweighted $l^1$ minimization method for stochastic elliptic
  differential equations.
\newblock {\em J. Comput. Phys.}, 248:87--108, 2013.

\bibitem{JJZhang2016}
Jiangjiang Zhang, Weixuan Li, Lingzao Zeng, and Laosheng Wu.
\newblock An adaptive gaussian process-based method for efficient bayesian
  experimental design in groundwater contaminant source identification
  problems.
\newblock {\em Water Resources Research}, 52(8):5971--5984, 2016.

\end{thebibliography}
\end{document}